\documentclass[a4paper, reqno]{amsart}

\usepackage[UKenglish]{babel}
\usepackage{amsmath, amssymb, amsthm}
\usepackage{amsaddr}
\usepackage{scrextend}
\usepackage[hidelinks]{hyperref}
\usepackage{xpatch}
\usepackage{color}
\usepackage{tikz-cd}
\usepackage{enumitem}
\usepackage{theoremref}
\usepackage{stmaryrd}
\usepackage{bussproofs}

\usepackage[inline,final]{showlabels}
\showlabels{thlabel}

\usepackage{imakeidx}
\makeindex[name=notes,title=Notes throughout the document (red text),columns=1]

\xpatchcmd{\proof}{\itshape}{\normalfont\bfseries}{}{}
\newtheoremstyle{repeat}{}{}{\itshape}{}{\bfseries}{.}{.5em}{#3, repeated}

\newtheorem{theorem}{Theorem}[section]
\newtheorem{proposition}[theorem]{Proposition}
\newtheorem{lemma}[theorem]{Lemma}
\newtheorem{corollary}[theorem]{Corollary}
\newtheorem{fact}[theorem]{Fact}

\theoremstyle{definition}
\newtheorem{definition}[theorem]{Definition}
\newtheorem{remark}[theorem]{Remark}
\newtheorem{convention}[theorem]{Convention}
\newtheorem{example}[theorem]{Example}
\newtheorem{question}[theorem]{Question}

\theoremstyle{repeat}
\newtheorem*{repeated-theorem}{Repeat}

\newcommand{\A}{\mathcal{A}}
\newcommand{\B}{\mathcal{B}}
\newcommand{\C}{\mathcal{C}}

\newcommand{\E}{\mathcal{E}}

\renewcommand{\L}{\mathcal{L}}

\newcommand{\R}{\mathbb{R}}
\newcommand{\U}{\mathcal{U}}

\newcommand{\Set}{\mathbf{Set}}
\newcommand{\FinSet}{\mathbf{FinSet}}

\newcommand{\Spec}{\mathbf{Spec}}
\newcommand{\Stone}{\mathbf{Stone}}
\newcommand{\KO}{\mathcal{K}^o}
\newcommand{\CohTheory}{\mathbf{CohTheory}}
\newcommand{\TypeSpaceFunc}{\mathbf{TypeSpaceFunc}}
\newcommand{\LTFunc}{\mathbf{LTFunc}}
\newcommand{\DLat}{\mathbf{DLat}}

\newcommand{\Topos}{\mathfrak{Topos}}
\newcommand{\CohTheoryBi}{\mathfrak{CohTheory}}
\newcommand{\TypeSpaceFuncBi}{\mathfrak{TypeSpaceFunc}}
\newcommand{\LTFuncBi}{\mathfrak{LTFunc}}

\newcommand{\Mod}{\text{-}\mathbf{Mod}}
\renewcommand{\O}{\mathcal{O}}

\DeclareMathOperator{\ctp}{ctp}
\DeclareMathOperator{\Hom}{Hom}
\DeclareMathOperator{\Th}{Th}
\renewcommand{\S}{\operatorname{S}}
\DeclareMathOperator{\SpecSpace}{Spec}
\DeclareMathOperator{\dom}{dom}

\newcommand{\op}{{\textup{op}}}
\renewcommand{\phi}{\varphi}
\renewcommand{\k}{\mathbf{k}}
\newcommand{\m}{\mathbf{m}}
\newcommand{\n}{\mathbf{n}}

\title{Type space functors and interpretations in positive logic}
\author{Mark Kamsma}
\address{University of East Anglia, Norwich, United Kingdom.}
\email[Mark Kamsma]{m.kamsma@uea.ac.uk}
\urladdr{https://markkamsma.nl}
\date{\today \\ \indent \emph{2020 Mathematics Subjects Classification}: Primary: 03C07; secondary: 03B20, 03G10, 03G30, 06E15}
 
\begin{document}

\begin{abstract}
We construct a 2-equivalence $\CohTheoryBi^\op \simeq \TypeSpaceFuncBi$. Here $\CohTheoryBi$ is the 2-category of positive theories and $\TypeSpaceFuncBi$ is the 2-category of type space functors. We give a precise definition of interpretations for positive logic, which will be the 1-cells in $\CohTheoryBi$. The 2-cells are definable homomorphisms. The 2-equivalence restricts to a duality of categories, making precise the philosophy that a theory is `the same' as the collection of its type spaces (i.e.\ its type space functor).

In characterising those functors that arise as type space functors, we find that they are specific instances of (coherent) hyperdoctrines. This connects two different schools of thought on the logical structure of a theory.

The key ingredient, the Deligne completeness theorem, arises from topos theory, where positive theories have been studied under the name of coherent theories.
\end{abstract}

\maketitle

\setcounter{tocdepth}{1}
\tableofcontents

\section{Introduction}
\label{sec:introduction}
For a first-order theory $T$, we can construct its Stone space $\S_n(T)$ of $n$-types. A function $f: \{1, \ldots, n\} \to \{1, \ldots, m\}$ then induces an open continuous function $\S_m(T) \to \S_n(T)$. These are standard constructions, and we can collect this data into a functor $\FinSet^\op \to \Stone$. In positive logic we consider only positive existential formulas. We will consider types of positive existential formulas in all models of a positive theory. The resulting type spaces are then spectral spaces. This setting was studied by Haykazyan in \cite{haykazyan_spaces_2019}.

We should mention a semantically different approach, namely the one taken by Ben-Yaacov in \cite{ben-yaacov_positive_2003, ben-yaacov_thickness_2003}. In that approach, only maximal positive existential types are considered. Equivalently, those are the positive existential types realised in an existentially closed model. We discuss this distinction in semantics further in \thref{rem:existentially-closed-positive-logic}.

The philosophy of type space functors is that they encode all information about a certain theory. Indeed, in \cite[Theorem 4.4]{haykazyan_spaces_2019} it is shown that $T$ and $T'$ are isomorphic (in the category that we call $\CohTheory$, \thref{def:2-category-cohtheory}) if and only if $\S(T)$ and $\S(T')$ are naturally isomorphic. This result will actually be corollary to our main theorem, from which it follows that $\S(-)$ is an equivalence of categories. 

Positive logic has also been studied in topos theory via classifying toposes (see section \ref{sec:connection-with-classifying-toposes}), where it is called \emph{coherent logic}. It is in fact from this world, that we get the key ingredient for our main proof: the Deligne completeness theorem (see \thref{fact:deligne-completeness-theorem}). The definition of type space functors is semantical: a type is a set of (positive existential) formulas that is realised in some model (in the classical sense, so in $\Set$). The characterisation we give is syntactical. The Deligne completeness theorem provides the link between these two.

The reader that is familiar with hyperdoctrines will recognise the characterisation of type space functors as specific (coherent) hyperdoctrines (see also \thref{rem:hyperdoctrines}). Hyperdoctrines were introduced by Lawvere \cite{lawvere_adjointness_2006} and have been studied in categorical settings (e.g.\ \cite{coumans_generalising_2012}). Of course, type spaces have long been studied by model theorists. This paper provides a precise proof that these two approaches are in fact the same. While the objects in both approaches are the same, the 1-cells and 2-cells in our 2-categories are based on model-theoretic ideas, and are different from the categorical treatments.

\vspace{\baselineskip}\noindent
\textbf{Main results.} We define the 2-category $\CohTheoryBi$ of coherent theories (i.e.\ positive theories), interpretations and definable homomorphisms. We also define the category $\CohTheory$ of coherent theories and strong interpretations (both in \thref{def:2-category-cohtheory}). To this end we formulate precisely what an interpretation is in positive logic (\thref{def:interpretation}). This generalises the usual first-order definition. In $\CohTheoryBi$, two theories will be equivalent precisely when they are bi-interpretable. Isomorphisms correspond to essentially having the same Morleyisation.

We also define the 2-category $\TypeSpaceFuncBi$ and category $\TypeSpaceFunc$ of \emph{type space functors} (\thref{def:2-category-type-space-func}). These are specific functors $\FinSet^\op \to \Spec$, where $\Spec$ denotes the category of spectral spaces (\thref{def:type-space-functor}).

The usual construction of a type space $\S_n(T)$ from a theory $T$ extends to a 2-functor between these 2-categories. This functor is part of a strict 2-equivalence, where the inverse 2-functor assigns a theory $\Th(F)$ to a type space functor $F$.
\begin{theorem}
\thlabel{thm:equivalence-2-categories}
The 2-functors
\[
\S: \CohTheoryBi^\op \rightleftarrows \TypeSpaceFuncBi: \Th
\]
form a 2-equivalence of 2-categories, which restricts to an equivalence
\[
\S: \CohTheory^\op \rightleftarrows \TypeSpaceFunc: \Th
\]
\end{theorem}
\begin{corollary}
\thlabel{cor:duality-first-order-theory-boolean-type-space}
The 2-equivalence in \thref{thm:equivalence-2-categories} restricts to positively model complete (i.e.\ first-order) theories and type space functors taking their values in $\Stone$.
\end{corollary}
There is a well-known duality $\DLat^\op \simeq \Spec$, where $\DLat$ is the category of distributive lattices. So we can view type space functors $\FinSet^\op \to \Spec$ equivalently as functors $\FinSet \to \DLat$, associating to $n$ the Lindenbaum-Tarski algebra of formulas in $n$ free variables. We call such functors \emph{Lindenbaum-Tarski functors}, and collect them in a 2-category $\LTFuncBi$ (\thref{def:category-of-hyperdoctrines}). This exhibits type space functors as specific instances of (coherent) hyperdoctrines (\thref{rem:hyperdoctrines}).
\begin{corollary}
\thlabel{cor:duality-theory-ltfunc}
There is a 2-equivalence of categories
\[
\CohTheoryBi \simeq \LTFuncBi,
\]
which restricts to an equivalence
\[
\CohTheory \simeq \LTFunc.
\]
This is induced by the equivalence in \thref{thm:equivalence-2-categories} and the duality $\DLat^\op \simeq \Spec$.
\end{corollary}

\vspace{\baselineskip}\noindent
\textbf{Overview.} We start with a brief recap of some preliminaries in section \ref{sec:preliminaries}: positive logic, spectral spaces and 2-categories. We only recall basic definitions, and refer to more extensive treatments for the interested reader.

The aim of section \ref{sec:category-of-theories-and-interpretations} is to define the 2-category $\CohTheoryBi$ and subsequently the category $\CohTheory$. For this we give a precise definition of an interpretation between theories from a syntactical point of view.

In section \ref{sec:category-of-type-space-functors} we first identify the characterising properties that a type space functor of the form $\S(T)$ always has. This gives rise to a 2-category $\TypeSpaceFuncBi$ of specific functors $\FinSet^\op \to \Spec$. We also define a category $\TypeSpaceFunc$, where the arrows are natural transformations satisfying the so-called \emph{Beck-Chevalley condition}. We then conclude that $\S(-)$ is a functor of 2-categories.

Section \ref{sec:hyperdoctrines-and-the-frobenius-condition} deals with the syntactical perspective of type space functors, and their link to hyperdoctrines. For the main theorem, only a technical result at the start of the section is needed (\thref{prop:frobenius-for-open-spectral-maps}).

Section \ref{sec:recovering-a-theory} describes how to construct a theory $\Th(F)$ from a type space functor $F$, as well as how to turn this operation into a 2-functor.

Section \ref{sec:the-2-equivalence} contains the proof of the main theorem.

We close out with section \ref{sec:connection-with-classifying-toposes}, linking $\CohTheoryBi$ to toposes via classifying toposes.

\vspace{\baselineskip}\noindent
\textbf{Acknowledgements.} I would like to thank my supervisor, Jonathan Kirby, for his feedback. Considering the 2-categorical setting was also his idea. I would also like to thank the anonymous referee for their feedback which helped improve the presentation of this paper, and for suggesting the equivalence in \thref{prop:beck-chevalley}. This paper is part of a PhD project at the UEA (University of East Anglia), and as such is supported by a scholarship from the UEA.
\section{Preliminaries}
\label{sec:preliminaries}

\subsection{Positive logic}
\label{subsec:positive-logic}
We only recall the basics of positive logic. There is a brief introduction in \cite{haykazyan_spaces_2019}, and more extensive treatments in \cite{poizat_positive_2018, ben-yaacov_fondements_2007, ben-yaacov_positive_2003}. We stick to single-sorted logic. The notation and conventions for multi-sorted logic are more complicated. However, there should be no intrinsic difficulty with generalising to multi-sorted logic.

Positive logic has been studied in topos theory under the name of \emph{coherent logic}. This section mentions both, so it can serve as a dictionary.
\begin{definition}
\thlabel{def:coherent-logic}
Fix a signature $\L$. A \emph{coherent formula} or \emph{positive existential formula} in $\L$ is one that is obtained from combining atomic formulas using $\wedge$, $\vee$, $\top$, $\bot$ and $\exists$. An \emph{h-inductive sentence} is a sentence of the form $\forall \bar{x}(\phi(\bar{x}) \to \psi(\bar{x}))$ where $\phi$ and $\psi$ are coherent formulas. A \emph{coherent theory} or \emph{positive theory} is a set of h-inductive sentences.
\end{definition}
\begin{convention}
\thlabel{conv:everything-is-coherent}
Whenever we say ``formula'' or ``theory'', we mean ``coherent formula'' and ``coherent theory'' respectively. We allow empty models. All semantics refer to the classical semantics (i.e.\ in $\Set$), unless explicitly stated otherwise.
\end{convention}
We can encode any function symbol as its graph using a relation symbol. Having only relation symbols simplifies things, and we lose no expressive power.
\begin{convention}
\thlabel{conv:relational-signature}
Every signature we consider is purely relational.
\end{convention}
Any first-order theory can be presented as a coherent theory. This is done through a process called \emph{Morleyisation}, see just after Fact 2.7 in \cite{haykazyan_spaces_2019}. We recall the following definition from there as well, giving a name to the coherent theories that are essentially first-order theories.
\begin{definition}
\thlabel{def:positively-model-complete}
A coherent theory $T$ is called \emph{positively model complete} if for every formula $\phi(x)$ there is a formula $\psi(x)$ such that $T \models \forall \bar{x}(\neg \phi(\bar{x}) \leftrightarrow \psi(\bar{x}))$.
\end{definition}
Recall that $T \models \varphi$ means that $\varphi$ is true in every model of $T$. We emphasise that we really are interested in all models of $T$, rather than just the existentially closed ones. We refer once more to \thref{rem:existentially-closed-positive-logic} for a discussion about the different semantics.
\subsection{Spectral spaces}
\label{subsec:spectral-spaces}
For an extensive treatment of spectral spaces, see \cite{dickmann_spectral_2019}.
\begin{definition}
\thlabel{def:compact-open-sets}
For a topological space $X$, write $\KO(X)$ for the partial order of compact open subsets. As such, we consider $\KO(X)$ as a category.
\end{definition}
\begin{definition}
\thlabel{def:sober-space}
Let $X$ be a topological space. A closed set $A \subseteq X$ is called \emph{irreducible} (sometimes also called \emph{hyperconnected}) if it cannot be written as the union of two proper closed subsets. We call $X$ a \emph{sober} space if every irreducible closed set is the closure of exactly one point.
\end{definition}
\begin{definition}
\thlabel{def:spectral-space}
A topological space $X$ is called \emph{spectral} if it is compact, $T_0$, sober and such that $\KO(X)$ is closed under finite intersections and forms a basis of the topology. A continuous function $f: X \to Y$ of spectral spaces is called a \emph{spectral map} if the preimage of a compact open set is compact open. We define $\Spec$ to be the category of spectral spaces and spectral maps.
\end{definition}
\begin{fact}[{\cite[1.3.4 and 1.3.20]{dickmann_spectral_2019}}]
\thlabel{fact:hausdorff-spectral-spaces}
A spectral space is $T_1$ iff it is Hausdorff iff it is a Stone space.
\end{fact}
We write $\DLat$ for the category of distributive lattices with  distributive lattice homomorphisms. There is a duality $\DLat^\op \simeq \Spec$, see for example \cite[Section II.3]{johnstone_stone_1992} or \cite[Section 3.2]{dickmann_spectral_2019}, called the Stone duality for distributive lattices.

Let $L$ be a distributive lattice, then the space associated to it is $\SpecSpace(L)$, the \emph{spectrum of $L$}. Its points are prime filters on $L$, and the sets
\[
\{ F \in \SpecSpace(L) : a \in F \},
\]
where $a \in L$, form a basis of compact open sets. A homomorphism $f: L \to L'$ of distributive lattices, induces a spectral map $\SpecSpace(L') \to \SpecSpace(L)$ by sending a prime filter $F'$ on $L'$ to $f^{-1}(F')$.

Conversely, given a spectral space $X$, the corresponding distributive lattice is given by $\KO(X)$ and a spectral map $f: X \to X'$ induces a homomorphism of distributive lattices $f^{-1}: \KO(X') \to \KO(X)$.
\subsection{2-categories}
\label{subsec:2-categories}
For an extensive treatment of 2-categories see \cite[Section B1.1]{johnstone_sketches_2002_v1} or \cite{johnson_2-dimensional_2020}. In this paper we will only be interested in strict 2-categories and strict 2-functors.
\begin{convention}
\thlabel{conv:2-category}
We use the font $\mathfrak{Category}$ for (strict) 2-categories, and $\mathbf{Category}$ for (normal) categories. We denote by $\mathfrak{Category}^\op$ the 2-category where all 1-cells are reversed, but the direction of the 2-cells does not change.
\end{convention}
In addition to objects (0-cells) and arrows (1-cells), a 2-category also has 2-cells, which are morphisms between 1-cells. This allows for the notion of \emph{equivalence} in a 2-category: two objects $A$ and $B$ are equivalent if there are 1-cells $f: A \to B$ and $g: B \to A$ such that their compositions are isomorphic, via 2-cells, to the identity. So for example, bi-interpretability of theories is the same thing as being equivalent in the appropriate 2-category, see \thref{rem:bi-interpretation-is-equivalence}.

Note that there are two different ways of composing 2-cells. The first, and most obvious, way is called \emph{vertical composition}: given 1-cells $f, g, h: A \to B$ and 2-cells $\alpha: f \Rightarrow g$ and $\beta: g \Rightarrow h$ we can form a 2-cell $\beta \alpha: f \Rightarrow h$. Then there is also \emph{horizontal composition}, which works as follows. Suppose we have a diagram as below:
\[
\begin{tikzcd}
A \arrow[rr, "f", bend left] \arrow[rr, "f'"', bend right] & \Downarrow \alpha & B \arrow[rr, "g", bend left] \arrow[rr, "g'"', bend right] & \Downarrow \beta & C
\end{tikzcd}
\]
Then the horizontal composition of $\alpha$ and $\beta$ yields a 2-cell $\beta * \alpha: gf \Rightarrow g'f'$. Of course, these compositions are subject to certain axioms, but these are not relevant for this paper and so we will not go into them here.

Throughout the paper we have simplified results to normal categories as well. This should help if the reader is less familiar with 2-categories.
\section{Category of theories and interpretations}
\label{sec:category-of-theories-and-interpretations}
\begin{definition}
\thlabel{def:interpretation}
Let $T$ and $T'$ be theories in signatures $\L$ and $\L'$ respectively. An \emph{interpretation} $(\Gamma, k): T \to T'$, where $k \geq 1$, assigns to each $n$-ary relation symbol $R$ in $T$ (including the binary symbol for equality) an $nk$-ary $\L'$-formula $\Gamma(R)$. This can be extended uniquely to a function sending $\L$-formulas to $\L'$-formulas. We then also require for all $\L$-formulas $\phi(\bar{x})$ and $\psi(\bar{x})$ that precisely have free variables $\bar{x}$ (i.e.\ they all appear in the formulas):
\[
T \models \forall \bar{x}(\phi(\bar{x}) \to \psi(\bar{x})) \implies
T' \models \forall \bar{x}_1 \ldots \bar{x}_n(\Gamma(\phi)(\bar{x}_1, \ldots, \bar{x}_n) \to \Gamma(\psi)(\bar{x}_1, \ldots, \bar{x}_n)).
\]
An interpretation $(\Gamma, 1): T \to T'$ that interprets equality as equality is called a \emph{strong interpretation}, and we simplify the notation to $\Gamma: T \to T'$.
\end{definition}
What we call strong interpretation is just called ``interpretation'' in \cite[Definition 4.3]{haykazyan_spaces_2019}. The same name is also used in topos theory for something in between our notions of strong interpretation and interpretation, see \thref{rem:different-notions-of-interpretation}. However, the model-theoretic notion of interpretation for first-order logic already exists and does in fact coincide with our definition for positively model complete theories, see e.g.\ \cite[Chapter 5]{hodges_model_1993}. 
\begin{remark}
\thlabel{rem:interpretation-models}
Let $(\Gamma, k): T \to T'$ be an interpretation and let $M \models T'$. Then we can define a model $\Gamma^*(M)$ as follows. The formula $\Gamma(x = y)$ defines an equivalence relation on $A = \{\bar{a} \in M^k : M \models \Gamma(x = x)(\bar{a})\}$. As underlying set for $\Gamma^*(M)$ we take $A / \Gamma(x = y)$. An $n$-ary relation symbol $R$ is then interpreted as
\[
\{([\bar{a}_1], \ldots, [\bar{a}_n]) \in \Gamma^*(M)^n : M \models \Gamma(R)(\bar{a}_1, \ldots, \bar{a}_n)\},
\]
where $[\bar{a}]$ denotes the equivalence class of $\bar{a}$. The fact that $(\Gamma, k)$ is an interpretation makes this well-defined. By induction we see that $\Gamma^*(M) \models \phi([\bar{a}_1], \ldots, [\bar{a}_n])$ if and only if $M \models \Gamma(\phi)(\bar{a}_1, \ldots, \bar{a}_n)$. So in particular, $\Gamma^*(M)$ is a model of $T$.

For a strong interpretation $\Gamma: T \to T'$, the underlying set of $\Gamma^*(M)$ is the same as for $M$ itself.
\end{remark}
\begin{remark}
\thlabel{rem:formulas-match-arity-trick}
With the notation of \thref{def:interpretation}: given any $\L$-formula $\phi(\bar{x})$, where \emph{not} all $\bar{x}$ appear, we can easily make them appear by considering the equivalent formula $\phi(\bar{x}) \wedge \bar{x} = \bar{x}$ (where $\bar{x} = \bar{x}$ is short for a conjunction of equalities). The point of course is that $\Gamma(\phi)$ and $\Gamma(\phi \wedge \bar{x} = \bar{x})$ may not be equivalent, because $\Gamma(x=x)$ may not be the entire domain.
\end{remark}
\begin{definition}
\thlabel{def:morphism-of-interpretations}
Let $(\Gamma, k), (\Gamma', k'): T \to T'$ be interpretations. A \emph{morphism of interpretations} $\theta: (\Gamma, k) \to (\Gamma', k')$ is a formula $\theta(x_1, \ldots, x_k, y_1, \ldots, y_{k'})$ in the language of $T'$ such that:
\begin{enumerate}[label=(\arabic*)]
\item $T' \models \forall \bar{x}(\Gamma(x = x)(\bar{x}) \to \exists \bar{y}\theta(\bar{x}, \bar{y}))$,
\item $T' \models \forall \bar{x}\bar{y}(\theta(\bar{x}, \bar{y}) \to \Gamma(x = x)(\bar{x}) \wedge \Gamma'(y = y)(\bar{y}))$,
\item $T' \models \forall \bar{x} \bar{x}' \bar{y}(\theta(\bar{x}, \bar{y}) \wedge \Gamma(x=x')(\bar{x}, \bar{x}') \to \theta(\bar{x}', \bar{y}))$,
\item $T' \models \forall \bar{x} \bar{y} \bar{y}'(\theta(\bar{x}, \bar{y}) \wedge \Gamma'(y=y')(\bar{y}, \bar{y}') \to \theta(\bar{x}, \bar{y}'))$,
\item for every formula $\phi(x_1, \ldots, x_n)$ in the language of $T$, we have:
\[
T' \models \forall \bar{x}_1 \ldots \bar{x}_n \bar{y}_1 \ldots \bar{y}_n(\Gamma(\phi)(\bar{x}_1, \ldots, \bar{x}_n) \wedge \bigwedge_{i = 1}^n \theta(\bar{x}_i, \bar{y}_i) \to \Gamma'(\phi)(\bar{y}_1, \ldots, \bar{y}_n)).
\]
This simplifies to $T' \models \Gamma(\phi) \to \Gamma'(\phi)$ for formulas without free variables.
\end{enumerate}
Two morphisms of interpretations are considered to be equal if they are equivalent formulas, modulo the relevant theory.

Given morphisms $\theta(\bar{x}, \bar{y}): (\Gamma, k) \to (\Gamma', k')$ and $\eta(\bar{y}, \bar{z}): (\Gamma', k') \to (\Gamma'', k'')$, we define their composition to be $\exists \bar{y}(\theta(\bar{x}, \bar{y}) \wedge \eta(\bar{y}, \bar{z}))$. One easily checks that this defines a morphism $(\Gamma, k) \to (\Gamma'', k'')$.
\end{definition}
\begin{remark}
\thlabel{rem:morphism-of-interpretations-gives-homomorphism-of-models}
Let $M$ be a model of $T'$ and use the notation from \thref{def:morphism-of-interpretations}. Then $\theta$ defines a homomorphism $f_\theta: \Gamma^*(M) \to \Gamma'^*(M)$. For an equivalence class $[\bar{a}] \in \Gamma^*(M)$ we let $[\bar{b}] \in \Gamma'^*(M)$ be such that $M \models \theta(\bar{a}, \bar{b})$, and we set $f_\theta([\bar{a}]) = [\bar{b}]$.
\end{remark}
\begin{definition}
\thlabel{def:2-category-cohtheory}
The \emph{2-category of theories} $\CohTheoryBi$ has as objects theories, as 1-cells interpretations and as 2-cells morphisms of interpretations. Horizontal composition is defined as follows. Consider the following diagram
\[
\begin{tikzcd}
T \arrow[rr, "{(\Gamma, k)}", bend left] \arrow[rr, "{(\Gamma', k')}"', bend right] & \Downarrow \theta & T' \arrow[rr, "{(\Delta, \ell)}", bend left] \arrow[rr, "{(\Delta', \ell')}"', bend right] & \Downarrow \eta & T''
\end{tikzcd}
\]
Then $\eta * \theta$ is the following formula:
\[
\exists \bar{y}_1 \ldots \bar{y}_{k'} \left( \Delta(\theta)(\bar{x}_1, \ldots, \bar{x}_k, \bar{y}_1, \ldots, \bar{y}_{k'}) \wedge \bigwedge_{i = 1}^{k'} \eta(\bar{y}_i, \bar{z}_i) \right),
\]
where each $\bar{x}_i$ and $\bar{y}_i$ is of length $\ell$ and each $\bar{z}_i$ is of length $\ell'$.

The \emph{category of theories} $\CohTheory$ has as objects theories and as arrows strong interpretations.
\end{definition}
Checking that $\CohTheoryBi$ satisfies the axioms of a strict 2-category is lengthy, but straightforward and we omit the proof.
\begin{remark}
\thlabel{rem:bi-interpretation-is-equivalence}
An equivalence of two theories $T$ and $T'$ in $\CohTheoryBi$ consists of interpretations $(\Gamma, k): T \to T'$ and $(\Gamma', k'): T' \to T$ such that both compositions are isomorphic, via 2-cells, to the relevant identity interpretations. Such 2-cells correspond to definable isomorphisms (see \thref{rem:morphism-of-interpretations-gives-homomorphism-of-models}). So equivalences in $\CohTheoryBi$ are precisely bi-interpretations.
\end{remark}
\section{Category of type space functors}
\label{sec:category-of-type-space-functors}
We start this section by recalling some definitions and facts. From \thref{def:type-space-functor} on we make and treat new definitions.
\begin{definition}
\thlabel{def:type}
Let $T$ be a theory, let $M$ be a model of $T$ and let $\bar{a} \in M$ be a tuple of elements. The \emph{(coherent) type of $\bar{a}$} is defined as:
\[
\ctp^M(\bar{a}) = \{ \phi(\bar{x}) : M \models \phi(\bar{a}) \}.
\]
A \emph{type} in free variables $\bar{x}$ is a set of formulas $p(\bar{x})$ such that $p(\bar{x}) = \ctp^M(\bar{a})$ for some model $M$ of $T$ and some $\bar{a} \in M$.
\end{definition}
The category of all natural numbers together with all maps between them is a skeleton of the category of finite sets $\FinSet$, so we can identify those categories.
\begin{definition}
\thlabel{def:type-space-functor-of-a-theory}
We define the \emph{type space functor} $\S(T): \FinSet^\op \to \Spec$ of a theory $T$ as follows.

Let $n \geq 0$ be a natural number. Then we define $\S_n(T)$ to be the set of all types in $n$ free variables. For a formula $\phi$ we write $[\phi] = \{ p \in S_n(T) : \phi \in p \}$, and we topologise $\S_n(T)$ by taking these sets as basic opens. We call $\S_n(T)$ the \emph{type space} of $T$ in $n$ free variables.

For a natural number $n$, we write $\n = \{1, \ldots, n\}$. For a function $f: \n \to \m$, we define a function $\S_f(T): \S_m(T) \to \S_n(T)$ as follows:
\[
\S_f(T)(p) = \{ \phi(x_1, \ldots, x_n) : \phi(x_{f(1)}, \ldots, x_{f(n)}) \in p(x_1, \ldots, x_m) \}.
\]
Semantically, for a realisation $\bar{a} \in M$ of $p$, we get $\S_f(T)(p) = \ctp^M(a_{f(1)}, \ldots, a_{f(n)})$.
\end{definition}
The elements in $\mathbf{n}$ are indices of variables, for which we prefer to start at $1$ rather than $0$. The more common notation for $\n$ would be $[n]$, but we do not want to confuse it with the notation $[\phi]$ for compact open set in a type space.
\begin{remark}
\thlabel{rem:existentially-closed-positive-logic}
In \cite{ben-yaacov_positive_2003, ben-yaacov_thickness_2003} type space functors in positive logic are also studied, but for different semantics. They only consider existentially closed models. Recall that a model $M$ of some theory $T$ is called \emph{existentially closed} (abbreviated as \emph{e.c.}) if for any $\bar{a} \in M$ and any formula $\phi(\bar{x})$ the following holds: whenever there is a homomorphism $f: M \to N$, where $N \models T$, such that $N \models \phi(f(\bar{a}))$ then already $M \models \phi(\bar{a})$.

The types realised in e.c.\ models are exactly the maximal types. So in this approach the following set of types is considered:
\[
M_n(T) = \{ \ctp^M(\bar{a}) : \bar{a} \text{ is an $n$-tuple in an e.c.\ model $M$ of $T$} \}.
\]
Note that we used a different notation from \cite{ben-yaacov_positive_2003, ben-yaacov_thickness_2003} to avoid clashes with our own notation ($\S_n(T)$ is used there too). This set $M_n(T)$ is then also topologised differently, namely by taking the sets $[\phi] = \{ p \in M_n(T) : \phi \in p \}$ as a basis of \emph{closed} sets, rather than open sets. This way we end up with a compact $T_1$ space, where the closed sets correspond to partial types (i.e.\ consistent sets of formulas).

A lot of model theory has been succesfully developed in this setting (e.g.\ \cite{ben-yaacov_positive_2003, ben-yaacov_thickness_2003, ben-yaacov_fondements_2007, poizat_positive_2018}). However, as Haykazyan notes in \cite{haykazyan_spaces_2019} this means that in the above setup there is no distinction between definable and type-definable sets.

Interestingly there is a link between the two approaches, as studied in the final two sections of \cite{haykazyan_spaces_2019}. For example, an omitting types theorem is proved for e.c.\ models, using the spaces $\S_n(T)$ rather than $M_n(T)$.
\end{remark}
\begin{fact}
\thlabel{fact:properties-type-space-functor-of-a-theory}
Let $T$ be a theory, then:
\begin{enumerate}[label=(\roman*)]
\item $U \subseteq \S_n(T)$ is compact and open iff $U = [\phi]$ for some formula $\phi$,
\item $T$ is positively model complete iff $\S_n(T)$ is a Stone space for all $n \geq 0$,
\item for any $f: \n \to \m$, the map $\S_f(T)$ is spectral and open.
\end{enumerate}
\end{fact}
\begin{proof}
Fact (i) follows from compactness and is \cite[Lemma 3.4]{haykazyan_spaces_2019}. For (ii): note that $T$ is  positively model complete if and only if $\S_n(T)$ is Hausdorff for all $n \geq 0$, by \cite[Proposition 3.8]{haykazyan_spaces_2019}. Which is equivalent to $\S_n(T)$ being a Stone space for all $n \geq 0$, by \thref{fact:hausdorff-spectral-spaces}. Finally, (iii) is precisely \cite[Proposition 4.1]{haykazyan_spaces_2019}.
\end{proof}
\begin{remark}
\thlabel{rem:explicit-description-preimage-direct-image}
In the proof of \cite[Proposition 4.1]{haykazyan_spaces_2019} an explicit description is given of the images and preimages of compact open sets. These explicit descriptions are useful, so we repeat them here:
\[
\S_f(T)([\phi(x_1, \ldots, x_m)]) = [\exists y_1 \ldots y_m(\phi(y_1, \ldots, y_m) \wedge x_1 = y_{f(1)} \wedge \ldots \wedge x_n = y_{f(n)})],
\]
and
\[
\S_f(T)^{-1}([\psi(x_1, \ldots, x_n)]) = [\psi(x_{f(1)}, \ldots, x_{f(n)})].
\]
Taking the direct image corresponds to introducing existential quantifiers and identifying variables. Taking the preimage corresponds to substitution of variables.
\end{remark}
\begin{remark}
\thlabel{rem:adjunctions-in-type-space-functor}
For every $f: \n \to \m$, we have a functor $\S_f(T)^{-1}: \KO(\S_n(T)) \to \KO(\S_m(T))$, because $\S_f(T)$ is a spectral map. We also have a functor $\S_f(T): \KO(\S_m(T)) \to \KO(\S_n(T))$ by taking direct images, because $\S_f(T)$ is an open continuous map. This gives rise to an adjoint pair of functors $\S_f(T) \dashv \S_f(T)^{-1}$.
\end{remark}
For a logical formula, it does not matter if we first quantify and identify variables and then substitute, or the other way around. This can be expressed by what is called the Beck-Chevalley condition (see e.g.\ \cite[Section IV.9]{maclane_sheaves_1994}). We will give a simplified definition for our situation.
\begin{definition}
\thlabel{def:beck-chevalley}
Suppose that we have a commuting square of open spectral maps as below on the left. Then we say that that square satisfies the \emph{Beck-Chevalley condition} if the induced square on the right commutes.
\[
\begin{tikzcd}
A \arrow[r, "f"] \arrow[d, "g"'] & B \arrow[d, "h"] \arrow[rd, "\Rightarrow" description, phantom] & \KO(A) \arrow[d, "g"'] & \KO(B) \arrow[l, "f^{-1}"'] \arrow[d, "h"] \\
C \arrow[r, "k"']                & D                                                               & \KO(C)                 & \KO(D) \arrow[l, "k^{-1}"]                
\end{tikzcd}
\]
\end{definition}
In the notation of \thref{def:beck-chevalley} we have for every $U \in \KO(B)$ that
\[
gf^{-1}(U) \subseteq gf^{-1}h^{-1}h(U) = gg^{-1}k^{-1}h(U) \subseteq k^{-1}h(U).
\]
Here we used the commutativity of the original square for the equality in the middle, because from that we immediately get $f^{-1}h^{-1} = g^{-1}k^{-1}$. We thus obtain the following fact.
\begin{fact}
\thlabel{fact:beck-chevalley-simplified}
In the notation of \thref{def:beck-chevalley}, the Beck-Chevalley condition is equivalent to having $k^{-1}h(U) \subseteq gf^{-1}(U)$ for all $U \in \KO(B)$.
\end{fact}
\begin{example}
\thlabel{ex:beck-chevalley-quantification}
Let $T$ be a theory. Let $f: \n \to \mathbf{n+1}$ be the inclusion. Let $g: \n \to \m$ be any function. Let $f': \m \to \mathbf{m+1}$ and $g': \mathbf{n+1} \to \mathbf{m+1}$ be the pushout of $f$ and $g$. That is, $f'$ is the inclusion and $g'$ is the extension of $g$ such that $g'(n+1) = m+1$. Then this induces a commutative square
\[
\begin{tikzcd}
\S_{m+1}(T) \arrow[r, "\S_{g'}(T)"] \arrow[d, "\S_{ f'}(T)"'] & \S_{n+1}(T) \arrow[d, "\S_f(T)"] \\
\S_m(T) \arrow[r, "\S_g(T)"']                                 & \S_n(T)                         
\end{tikzcd}
\]
This square satisfies the Beck-Chevalley condition, which comes down to having
\[
[\exists x_{n+1} \phi(x_{g(1)}, \ldots, x_{g(n)}, x_{n+1})] =
[\exists x_{m+1} \phi(x_{g'(1)}, \ldots, x_{g'(n)}, x_{g'(n+1)})],
\]
for every formula $\phi(x_1, \ldots, x_{n+1})$. Indeed, this equality holds because the right side is just
\[
[\exists x_{m+1} \phi(x_{g(1)}, \ldots, x_{g(n)}, x_{m+1})]
\]
and $m+1$ is not in the range of $g$.
\end{example}
\begin{example}
\thlabel{ex:beck-chevalley-equality}
Let $T$ be a theory. Let $f: \mathbf{n+1} \to \n$ be the identity on $\n$ and $f(n+1) = n$. Let $g: \mathbf{n+1} \to \m$ be any function. Denote by $f': \m \to \k$ and $g': \n \to \k$ the pushout of $f$ and $g$. Again, this pushout induces a commutative square like in \thref{ex:beck-chevalley-quantification}, and again this square satisfies the Beck-Chevalley condition. In this case it comes down to the equality
\begin{align*}
&[\phi(x_{g(1)}, \ldots, x_{g(n)}) \wedge x_{g(n)} = x_{g(n+1)}] &= \\
&[\exists y_1 \ldots y_k (\phi(y_{g'(1)}, \ldots, y_{g'(n)}) \wedge x_1 = y_{f'(1)} \wedge \ldots \wedge x_m = y_{f'(m)}) ],
\end{align*}
for every formula $\phi(x_1, \ldots, x_n)$.

By \thref{fact:beck-chevalley-simplified} we only have to check that the first formula implies the second. So we let $x_1, \ldots, x_m$ be such that $\phi(x_{g(1)}, \ldots, x_{g(n)})$ and $x_{g(n)} = x_{g(n+1)}$. Define $\alpha: \n \to \{x_1, \ldots, x_m\}$ by $\alpha(i) = x_{g(i)}$, and let $\beta: \m \to \{x_1, \ldots, x_m\}$ be given by $\beta(i) = x_i$. Then since $x_{g(n)} = x_{g(n+1)}$ we have that $\alpha f = \beta g$, so by the universal property of the pushout there is $\gamma: \k \to \{x_1, \ldots, x_m\}$. We take $y_i = \gamma(i)$ for all $1 \leq i \leq k$, then one easily verifies that $\phi(y_{g'(1)}, \ldots, y_{g'(n)}) \wedge x_1 = y_{f'(1)} \wedge \ldots \wedge x_m = y_{f'(m)}$ holds, as required.
\end{example}
\begin{proposition}
\thlabel{prop:type-space-functor-beck-chevalley}
Let $T$ be a theory. Any pushout square in $\FinSet$ induces a commutative square under the image of $\S(T)$, which satisfies the Beck-Chevalley condition.
\end{proposition}
\begin{proof}
In \thref{ex:beck-chevalley-quantification} and \thref{ex:beck-chevalley-equality} we have verified this claim for specific functions $f$. Up to isomorphism, every function in $\FinSet$ decomposes as such functions. So by composing squares we conclude that the claim indeed holds for any pushout square.
\end{proof}
We can now characterise those functors that arise as type space functors.
\begin{convention}
\thlabel{conv:functor-finset-spec}
Let $F: \FinSet^\op \to \Spec$ be a functor, we denote by $F_n$ and $F_f$ the images of $n$ and $f: \n \to \m$ respectively, to match the notation of $\S(T)$.
\end{convention}
\begin{definition}
\thlabel{def:type-space-functor}
A functor $F: \FinSet^\op \to \Spec$ is called \emph{open} if every map in its image is open. Such an open functor is called a \emph{type space functor} if for every pushout in $\FinSet$, the induced square satisfies the Beck-Chevalley condition.
\[
\begin{tikzcd}
\mathbf{a} \arrow[rrd, "\text{pushout}", phantom] &  & \mathbf{b} \arrow[ll, "f"'] \arrow[rd, "\Rightarrow", phantom] & F_a \arrow[rrr, "F_f"] \arrow[d, "F_g"'] \arrow[rrrd, "\text{satisfies B.C.}", phantom] &  &  & F_b \arrow[d, "F_h"] \arrow[rd, "\text{i.e.}", phantom] & \KO(F_a) \arrow[d, "F_g"'] \arrow[rd, "\text{commutes}", phantom] & \KO(F_b) \arrow[l, "F_f^{-1}"'] \arrow[d, "F_h"] \\
\mathbf{c} \arrow[u, "g"]                         &  & \mathbf{d} \arrow[ll, "k"] \arrow[u, "h"']                     & F_c \arrow[rrr, "F_k"']                                                                 &  &  & F_d                                                     & \KO(F_c)                                                          & \KO(F_d) \arrow[l, "F_k^{-1}"]                  
\end{tikzcd}
\]
\end{definition}
The discussion below is interesting in its own right, but not relevant for the rest of the paper, so the reader may skip to after \thref{ex:beck-chevalley-non-surjective} if they so wish. In \cite[Definition 2.18]{ben-yaacov_positive_2003} there is a condition called the ``amalgamation property'' which is related to the Beck-Chevalley condition. Recall that in their setup only maximal types are considered (see \thref{rem:existentially-closed-positive-logic}). If we restrict our setup to consider theories where all types are maximal, we get precisely the theories whose type spaces are $T_1$ and hence Stone spaces (see \thref{fact:hausdorff-spectral-spaces}). In that case we get that the two conditions are equivalent. We provide a proof below and thank the anonymous referee for pointing out this equivalence. Afterwards, in \thref{ex:beck-chevalley-non-surjective}, we show that the two conditions are generally not equivalent in our setting.
\begin{proposition}
\thlabel{prop:beck-chevalley}
Suppose we are given a commuting square of open spectral maps between Stone spaces as follows:
\[
\begin{tikzcd}
A \arrow[r, "f"] \arrow[d, "g"'] & B \arrow[d, "h"] \\
C \arrow[r, "k"']                & D               
\end{tikzcd}
\]
Then the following are equivalent:
\begin{enumerate}[label=(\roman*)]
\item the universal map $u: A \to B \times_D C$ is surjective;
\item the square satisfies the Beck-Chevalley condition.
\end{enumerate}
\end{proposition}
\begin{proof}
We recall that $B \times_D C = \{(b, c) \in B \times C : h(b) = k(c)\}$ and $u(a) = (f(a), g(a))$.

First we prove (i) $\implies$ (ii). Let $U \in \KO(B)$. By \thref{fact:beck-chevalley-simplified} we only need to prove $k^{-1}h(U) \subseteq gf^{-1}(U)$. So let $c \in k^{-1}h(U)$. Because we have $k(c) \in h(U)$ there must be $b \in U$ such that $h(b) = k(c)$. Then $(b, c) \in B \times_D C$, so by surjectivity there is $a \in A$ with $u(a) = (b, c)$. This means that $f(a) = b$ and $g(a) = c$. As $b \in U$ we thus have $a \in f^{-1}(U)$ and hence $c \in gf^{-1}(U)$, as required.

Now we prove the other direction, (ii) $\implies$ (i). Let $(b, c) \in B \times_D C$. We claim that for any $b \in B_0 \in \KO(B)$ and $c \in C_0 \in \KO(C)$ the set $f^{-1}(B_0) \cap g^{-1}(C_0)$ is nonempty.

Let $B_0$ and $C_0$ be as in the claim. Then we have $k(c) = h(b) \in h(B_0)$, so $c \in k^{-1}h(B_0)$. By assumption we have $k^{-1}h(B_0) = gf^{-1}(B_0)$, so $c \in gf^{-1}(B_0)$. Let $a \in f^{-1}(B_0)$ such that $g(a) = c$. Then $a \in g^{-1}(c) \subseteq g^{-1}(C_0)$, which proves the claim.

We set
\[
\U = \{f^{-1}(B_0) : b \in B_0 \in \KO(B)\} \cup \{g^{-1}(C_0) : c \in C_0 \in \KO(C)\}.
\]
It then follows from the claim, together with the fact that $\KO(B)$ and $\KO(C)$ are closed under finite intersections, that $\U$ has the finite intersection property. That is, any finite collection of sets in $\U$ has nonempty intersection. Because we assumed the spaces involved to be Stone spaces, the sets in $\KO(A)$ are clopen. So $\bigcap \U$ is nonempty and we find $a \in \bigcap \U$. By construction we then have that $f(a) \in B_0$ for every $b \in B_0 \in \KO(B)$. It then follows from the $T_1$ separation axiom (again, using that we are dealing with Stone spaces), together with the fact that $\KO(B)$ is a basis for the topology on $B$, that $f(a) = b$. Similarly we get $g(a) = c$. So $u(a) = (b, c)$, and we conclude that $u$ is indeed surjective.
\end{proof}
In the above proof we used the assumption that the spaces are Stone spaces in two places. The first was to produce an element in $\bigcap \U$, which can in fact be done in any spectral space, see \cite[Theorem 1.3.14]{dickmann_spectral_2019}. The second use, namely of the $T_1$ separation axiom, is essential, as shown in the example below. Note that this also means that the direction (i) $\implies$ (ii) goes through for any spectral spaces.
\begin{example}
\thlabel{ex:beck-chevalley-non-surjective}
We will show that the requirement of the spaces being Stone spaces cannot be dropped in \thref{prop:beck-chevalley}. More precisely, we will construct a theory $T$ and a pushout in $\FinSet$ such that the image of that pushout under $\S(T)$ will not satisfy the surjectivity condition in \thref{prop:beck-chevalley}(i). Recall that on the other hand that square will satisfy the Beck-Chevalley condition, as shown in \thref{prop:type-space-functor-beck-chevalley}.

Consider the language with three unary predicates $P(x)$, $Q(x)$ and $R(x)$. Let $T$ consist of just the h-inductive sentence $\forall xy(P(x) \wedge Q(y) \to R(x) \vee R(y))$. We let $M_1 = \{a, c\}$ be such that $P(M_1) = \{a, c\}$ and $Q(M_1) = R(M_1) = \{c\}$, and we let $M_2 = \{b, c\}$ with $Q(M_2) = \{b, c\}$ and $P(M_2) = R(M_2) = \{c\}$. Then $M_1$ and $M_2$ are models of $T$. There is a maximal zero-type, which is isolated by $\exists x(P(x) \wedge Q(x) \wedge R(x))$. So this is indeed the zero-type that is realised in $M_1$ and $M_2$.

Consider the pushout of $\mathbf{1} \leftarrow \mathbf{0} \to \mathbf{1}$, which is the disjoint union $\mathbf{1} + \mathbf{1} =\mathbf{2}$. We will show that the universal map $u: \S_2(T) \to \S_1(T) \times_{\S_0(T)} \S_1(T)$ is not surjective. Set $p(x) = \ctp^{M_1}(a)$ and $q(y) = \ctp^{M_2}(b)$. Then $p(x), q(y) \in \S_1(T)$ while their restrictions to the zero-type are the same, as we argued before. So $(p(x), q(y)) \in \S_1(T) \times_{\S_0(T)} \S_1(T)$. We claim that there can be no $r(x, y) \in \S_2(T)$ with $u(r(x, y)) = (p(x), q(y))$. Suppose that such an $r(x, y)$ is realised by some $a', b'$ in some model $M$ we would have $M \models P(a') \wedge Q(b')$. Then by the one h-inductive sentence that we have in $T$ we get $M \models R(a') \vee R(b')$, but $R(x) \not \in p(x)$ and $R(y) \not \in q(y)$. So $r(x, y)$ cannot restrict to both $p(x)$ and $q(y)$. We have arrived at our contradiction and conclude that $u$ is not surjective.
\end{example}
\begin{convention}
\thlabel{conv:function-finite-sets-multiplied}
For a function of finite sets $f: \n \to \m$ we use $f^{\times k}: \n \k \to \m \k$ to denote the map that is $f$ on each of the $k$ copies of $\n$. More precisely, viewing $\n\k$ as $\n \times \k$ (and similarly for $\m\k$), we have $f^{\times k}(a, b) = (f(a), b)$ for any $a \in \n$ and $b \in \k$.
\end{convention}
\begin{definition}
\thlabel{def:cartesian-family}
Let $F$ be a type space functor. For $k \geq 1$, a \emph{cartesian family (of arity $k$)} is a family $\Theta_n \subseteq F_{nk}$ of compact open sets, indexed by the natural numbers, such that the following holds. For $1 \leq i \leq n$ denote by $j_{i,n}: 1 \hookrightarrow \n$ the map with value $i$. For all $n$, we require:
\[
\Theta_n = \bigcap_{i = 1}^n F_{j_{i,n}^{\times k}}^{-1}(\Theta_1).
\]
For $n = 0$ this is the empty intersection, that is $\Theta_0 = F_0$.
\end{definition}
The name cartesian family is based on the idea that compact open sets correspond to definable sets (\thref{fact:properties-type-space-functor-of-a-theory}(i)). Then $\Theta_n$ corresponds to the cartesian product of $n$ times $\Theta_1$. The entire cartesian family can clearly be recovered from $\Theta_1$.
\begin{lemma}
\thlabel{lem:cartesian-family-interaction-maps}
Let $\Theta_n$ be a cartesian family of arity $k$ for $F$ and let $f: \n \to \m$ be any function. Then $\Theta_m \subseteq F_{f^{\times k}}^{-1}(\Theta_n)$ and if $f$ is surjective we have equality.
\end{lemma}
\begin{proof}
We have
\[
F_{f^{\times k}}^{-1}(\Theta_n) =
F_{f^{\times k}}^{-1}(\bigcap_{i = 1}^n F_{j_{i,n}^{\times k}}^{-1}(\Theta_1)) =
\bigcap_{i = 1}^n F_{f^{\times k}}^{-1} F_{j_{i,n}^{\times k}}^{-1}(\Theta_1) =
\bigcap_{i = 1}^n F_{j^{\times k}_{f(i), m}}^{-1}(\Theta_1).
\]
The result follows because $\Theta_m = \bigcap_{i = 1}^m F_{j_{i,m}^{\times k}}^{-1}(\Theta_1)$, with the final remark because $f^{\times k}$ is surjective precisely when $f$ is.
\end{proof}
\begin{definition}
\thlabel{def:partial-natural-transformation}
Let $F, F'$ be type space functors. A \emph{partial natural transformation (of arity $k$)} is a pair $(\beta, k): F \dashrightarrow F'$ where $k \geq 1$ and $\beta$ is a family of partial spectral maps $\beta_n: F_{nk} \dashrightarrow F'_n$. Here $n$ ranges over the natural numbers. The domains of $\beta$ are required to form a cartesian family (of arity $k$). Furthermore, for each $f: \n \to \m$ we require the following to commute:
\[
\begin{tikzcd}
F_{mk} \arrow[r, "\beta_m", dashed] \arrow[d, "F_{f^{\times k}}"'] & F_m' \arrow[d, "F_f'"] \\
F_{nk} \arrow[r, "\beta_n"', dashed]                               & F_n'                  
\end{tikzcd}
\]
\end{definition}
This last requirement is well-defined, because \thref{lem:cartesian-family-interaction-maps} implies that we always have $F_{f^{\times k}}(\dom(\beta_m)) \subseteq \dom(\beta_n)$.

The partial natural transformations will correspond to interpretations, so we need to express that they preserve necessary logical properties. This is done using a weak version of the Beck-Chevalley condition. \thref{ex:beck-chevalley-fail-for-partial-natural-transformation} illustrates why we have to consider this weaker version.
\begin{definition}
\thlabel{def:partial-natural-transformation-beck-chevalley}
A partial natural transformation $(\beta, k): F \to F'$ is said to satisfy the \emph{weak Beck-Chevalley condition} if for any $f: \n \to \m$ the following commutes:
\[
\begin{tikzcd}
\KO(F'_n) \arrow[rrr, "\beta_n^{-1} \cap F_{f^{\times k}}(F_{mk})"] &  &  & \KO(F_{nk})                                \\
\KO(F'_m) \arrow[u, "F'_f"] \arrow[rrr, "\beta_m^{-1}"']            &  &  & \KO(F_{mk}) \arrow[u, "F_{f^{\times k}}"']
\end{tikzcd}
\]
We say that $(\beta, k)$ satisfies the \emph{Beck-Chevalley condition} if the above diagram commutes already without taking the intersection. That is $\beta_n^{-1} F'_f = F_{f^{\times k}} \beta_m^{-1}$.
\end{definition}
Finally, we need something that corresponds to morphisms of interpretations. So these will be the 2-cells in the 2-category of type space functors. We will essentially let the equivalence in our main result (\thref{thm:equivalence-2-categories}) induce this 2-category structure. Not much attention will be given afterwards to the 2-cells, so the reader is welcome to use this `definition'. In practice we have to be careful about circularity, so formally we proceed as follows. 

A morphism of interpretations is given by a formula. So this will correspond to a compact open set in a type space functor. We will then also need to code the conditions (1)--(5) from \thref{def:morphism-of-interpretations}. In \thref{def:theory-of-type-space-functor} we develop what we call the internal logic of a type space functor. This is a way to talk about the compact open sets in a type space functor as if they were formulas.
\begin{definition}
\thlabel{def:morphism-partial-natural-transformations}
Let $(\beta, k), (\beta', k'): F \dashrightarrow F'$ be partial natural transformations. A \emph{morphism of partial natural transformations} $\Theta: (\beta, k) \to (\beta', k')$ is a compact open set $\Theta \in \KO(F_{k+k'})$ satisfying the equivalent to (1)--(5) from \thref{def:morphism-of-interpretations}, encoded in the internal logic of $F$. Vertical and horizontal composition are then also defined as for morphisms of interpretations, encoded in the internal logic of $F$.
\end{definition}
Even though we have not treated the internal logic of a type space functor yet, the following should give an idea of how to apply it in the above definition. For example, the internal logic version of \thref{def:morphism-of-interpretations}(5) becomes:
\[
\left \llbracket R_{\beta_n^{-1}(U)}(\bar{x}_1, \ldots, \bar{x}_n) \wedge \bigwedge_{i = 1}^n R_\Theta(\bar{x}_i, \bar{y}_i) \right \rrbracket \subseteq \left \llbracket R_{\beta_n'^{-1}(U)}(\bar{y}_1, \ldots, \bar{y}_n) \right \rrbracket,
\]
for all $n$ and all $U \in \KO(F_n')$.

\begin{definition}
\thlabel{def:the-rest-of-s}
For an interpretation $(\Gamma, k): T \to T'$, we define the partial natural transformation $\S(\Gamma, k): \S(T') \dashrightarrow \S(T)$ of arity $k$, where $\S(\Gamma, k)_n: \S_{nk}(T') \dashrightarrow \S_n(T)$ has domain $[\Gamma(\bar{x} = \bar{x})]$ and is given by
\[
\S(\Gamma, k)_n(p) = \{\phi(\bar{x}) : \Gamma(\phi(\bar{x}) \wedge \bar{x} = \bar{x}) \in p\}.
\]
For a morphism of interpretations $\theta: (\Gamma, k) \to (\Gamma', k')$ we define the morphism of partial natural transformations $\S(\theta): \S(\Gamma, k) \to \S(\Gamma', k')$ by
\[
\S(\theta) = [\theta(\bar{x}, \bar{y})] \subseteq \S_{k+k'}(T').
\]
\end{definition}
\begin{remark}
\thlabel{rem:semantic-s-interpretation}
Using \thref{rem:interpretation-models}, we can also define $\S(\Gamma, k)$ semantically as follows. For $p \in [\Gamma(\bar{x} = \bar{x})] \subseteq \S_{nk}(T')$, let $\bar{a}_1, \ldots, \bar{a}_n$ be a realisation of $p$ in some $M \models T'$, then $\S(\Gamma, k)_n(p) = \ctp^{\Gamma^*(M)}([\bar{a}_1], \ldots, [\bar{a}_n])$.
\end{remark}
\begin{lemma}
\thlabel{lem:properties-functor-s}
With the notation as in \thref{def:the-rest-of-s}, we have:
\begin{enumerate}[label=(\roman*)]
\item $\S(\Gamma, k)_n^{-1}([\phi(\bar{x})]) = [\Gamma(\phi(\bar{x}) \wedge \bar{x} = \bar{x})]$,
\item $\S(\Gamma, k)$ satisfies the weak Beck-Chevalley condition.
\end{enumerate}
For a strong interpretation $\Gamma: T \to T'$ this simplifies to:
\begin{enumerate}[label=(\roman*')]
\item $\S(\Gamma)_n^{-1}([\phi]) = [\Gamma(\phi)]$,
\item $\S(\Gamma)$ satisfies the Beck-Chevalley condition.
\end{enumerate}
\end{lemma}
\begin{proof}
This comes down to writing out definitions, where (i) is useful for (ii).
\end{proof}
\begin{definition}
\thlabel{def:2-category-type-space-func}
The \emph{2-category of type space functors} $\TypeSpaceFuncBi$ has as objects type space functors, as 1-cells partial natural transformations that satisfy weak Beck-Chevalley and as 2-cells morphisms of partial natural transformations.

The \emph{category of type space functors} $\TypeSpaceFunc$ has as objects type space functors and as arrows natural transformations that satisfy Beck-Chevalley.
\end{definition}
\begin{corollary}
\thlabel{cor:functor-s}
The operations defined in \thref{def:type-space-functor-of-a-theory} and \thref{def:the-rest-of-s} define a 2-functor $\S: \CohTheoryBi^\op \to \TypeSpaceFuncBi$, which restricts to a functor $\S: \CohTheory^\op \to \TypeSpaceFunc$.
\end{corollary}
\begin{proof}
Let $T$ be a theory, by \thref{fact:properties-type-space-functor-of-a-theory} and \thref{prop:type-space-functor-beck-chevalley}, $\S(T)$ is actually an object in $\TypeSpaceFuncBi$ and in $\TypeSpaceFunc$.

For an interpretation $(\Gamma, k): T \to T'$, \thref{lem:properties-functor-s} implies that $\S(\Gamma, k)$ is indeed a partial natural transformation. By the same lemma, a strong interpretation will indeed correspond to an arrow in $\TypeSpaceFunc$.

The requirements on a morphism of partial natural transformations just code those on a morphism of partial natural transformations. So the 2-functor is also well-defined on 2-cells.
\end{proof}
\begin{example}
\thlabel{ex:beck-chevalley-fail-for-partial-natural-transformation}
This example illustrates why for general interpretations we have to consider the weak Beck-Chevalley condition. Let $T$ be the empty theory in the empty language. Let $T'$ be the theory with one binary relation symbol $E$, expressing that $E$ is an equivalence relation (possibly not defined everywhere). We define $(\Gamma, 1): T \to T'$ by setting $\Gamma(x = y)$ to be $E(x, y)$. Let $f: \mathbf{2} \to \mathbf{1}$, then writing out definitions we have
\[
\S(\Gamma, 1)_2^{-1}(\S_f(T)([x = x])) = \S(\Gamma, 1)_2^{-1}([x = y]) = [E(x, y)],
\]
and
\[
\S_f(T')(\S(\Gamma, 1)_1^{-1}([x = x])) = \S_f(T')([E(x, x)]) = [x = y \wedge E(x, x) \wedge E(y, y)].
\]
Clearly these two are not the same. The first one is missing a conjunction with ``$x = y$'', which would precisely correspond to intersecting with the image of $S_f(T')$.
\end{example}
\section{Hyperdoctrines and the Frobenius condition}
\label{sec:hyperdoctrines-and-the-frobenius-condition}
In this section we connect type space functors with (coherent) hyperdoctrines. For this we recall one more property, the Frobenius condition, and we show that type space functors satisfy this condition (\thref{def:frobenius-condition} and \thref{prop:frobenius-for-open-spectral-maps}). Besides that the contents of this section are not needed for the main theorem.

We always have that $\psi(x_1) \wedge \exists x_2 \phi(x_1, x_2)$ is equivalent to $\exists x_2 (\psi(x_1) \wedge \phi(x_1, x_2))$. If we let $f: \mathbf{1} \to \mathbf{2}$, then this can be expressed as
\[
\S_f(T)([\phi] \cap \S_f(T)^{-1}([\psi])) = \S_f(T)([\phi]) \cap [\psi].
\]
This property makes sense for arbitrary adjoints between categories with products, and is called the Frobenius condition (see e.g.\ \cite[Section IV.9]{maclane_sheaves_1994}). We will once more just be interested in the simpler case where the categories involved are distributive lattices.
\begin{definition}
\thlabel{def:frobenius-condition}
Let $\A$ and $\B$ be distributive lattices, viewed as categories. Let $F: \A \rightleftarrows \B: G$ be functors, with $F \dashv G$. Then this pair of adjoints is said to satisfy the \emph{Frobenius condition} if for all $A$ and $B$ we have:
\[
F(A \wedge G(B)) = F(A) \wedge B.
\]
\end{definition}
In the notation of \thref{def:frobenius-condition} we have $F(A \wedge G(B)) \leq F(A)$ and $F(A \wedge G(B)) \leq FG(B) \leq B$. Combining this we arrive at the following fact.
\begin{fact}
\thlabel{fact:frobenius-automatic-direction}
For any pair of adjoint functors as in \thref{def:frobenius-condition}, we always have $F(A \wedge G(B)) \leq F(A) \wedge B$.
\end{fact}
For every $f: \n \to \m$ in $\FinSet$ the adjoints $\S_f(T) \dashv \S_f^{-1}(T)$ satisfy the Frobenius condition. This can be checked directly, but it also follows from the following proposition.
\begin{proposition}
\thlabel{prop:frobenius-for-open-spectral-maps}
Let $f: X \to Y$ be an open spectral map of spectral spaces, then the corresponding homomorphism of distributive lattices $f^{-1}: \KO(Y) \to \KO(X)$ has a left adjoint $f: \KO(X) \to \KO(Y)$ by taking the image under $f$. Furthermore, this pair $f \dashv f^{-1}$ satisfies the Frobenius condition.
\end{proposition}
\begin{proof}
As before, taking the direct image is left adjoint to taking the preimage. By \thref{fact:frobenius-automatic-direction} it is then enough to check $f(U) \cap V \subseteq f(U \cap f^{-1}(V))$ for all $U \in \KO(X)$ and $V \in \KO(Y)$, which is easily done.
\end{proof}
Using the duality between $\Spec$ and $\DLat$ we may also view type space functors as functors $\FinSet \to \DLat$, giving a more syntactic perspective. To characterise these functors from this perspective, we prove the converse of \thref{prop:frobenius-for-open-spectral-maps}.
\begin{proposition}
\thlabel{prop:dlat-left-adjoint-and-frobenius-implies-open}
Let $f: A \to B$ be a homomorphism of distributive lattices with a left adjoint $h: B \to A$ satisfying the Frobenius condition. Then the corresponding spectral map $f^{-1}: \SpecSpace(B) \to \SpecSpace(A)$ is open.
\end{proposition}
\begin{proof}
Let $U$ be a basic open in $\SpecSpace(B)$, so $U = \{F \in \SpecSpace(B) : b \in F\}$ for some $b \in B$. We claim that $U' = \{f^{-1}(F) : F \in U\}$ is the basic open $V = \{G \in \SpecSpace(A) : h(b) \in G\}$, which would prove the proposition.

First, we prove $U' \subseteq V$. We claim that for a prime filter $F \subseteq B$, we have that
\[
f^{-1}(F) = \uparrow h(F) := \{ a \in A : h(b') \leq a \text{ for some } b' \in F\}.
\]
For $a \in f^{-1}(F)$ we have that $f(a) \in F$, so the unit of the adjunction gives $hf(a) \leq a$ and thus $a \in \;\uparrow h(F)$. For the other inclusion, let $a \in \;\uparrow h(F)$ and $b' \in F$ such that $h(b') \leq a$. The adjunction gives $b' \leq f(a)$, and thus $f(a) \in F$, so $a \in f^{-1}(F)$.

Let $F \in U$, then $b \in F$, so $h(b) \in \;\uparrow h(F) = f^{-1}(F)$ using the claim. We thus conclude that indeed $f^{-1}(F) \in V$.

It remains to prove that $V \subseteq U'$. So let $G$ be a prime filter containing $h(b)$. Then for any $a \in G$ we have $f(a) \wedge b \neq 0$. Because if $f(a) \wedge b = 0$, then by Frobenius we would have $0 = h(f(a) \wedge b) = a \wedge h(b) \in G$. So if we close $f(G) \cup \{b\}$ under finite meets and then take the upward closure, we get a filter $F'$.

Consider
\[
I = \{b' \in B : b' \leq f(a) \text{ for some } a \not \in G\},
\]
then $I$ is an ideal. Indeed, it is clearly non-empty and downwards closed. For $b_1 \leq f(a_1)$ and $b_2 \leq f(a_2)$ with $a_1, a_2 \not \in G$, we have $a_1 \vee a_2 \not \in G$ since $G$ is a prime filter, and indeed $b_1 \vee b_2 \leq f(a_1) \vee f(a_2) = f(a_1 \vee a_2)$.

We claim that $I \cap F' = \emptyset$. Any $b' \in I \cap F'$ would satisfy $f(a) \wedge b \leq b' \leq f(a')$ for some $a \in G$ and $a' \not \in G$. So applying the adjunction and Frobenius we would have $a \wedge h(b) = h(f(a) \wedge b) \leq a'$, which implies $a' \in G$ and so we have a contradiction.

Using Zorn's lemma, we can then extend $F'$ to a filter $F$ that is maximal disjoint from $I$. In particular, such an $F$ is a prime filter containing $b$. We finish our proof by showing that $f^{-1}(F) = G$. We directly have $G \subseteq f^{-1}(f(G)) \subseteq f^{-1}(F)$. For the other inclusion we let $a \in f^{-1}(F)$, then $f(a) \in F$ and thus $f(a) \not \in I$. So we must have $a \in G$, as required, because otherwise we would have $f(a) \in I$.
\end{proof}
\begin{corollary}
\thlabel{cor:spectral-map-open-iff-dlat-map-left-adjoint-and-frobenius}
A spectral map of spectral spaces is open if and only if the corresponding homomorphism of distributive lattices has a left adjoint satisfying the Frobenius condition.
\end{corollary}
If we take the perspective of $\S(T)$ being a functor into $\DLat$, then we get a functor that associates to each $n$ the Lindenbaum-Tarski algebra of coherent formulas in $n$ free variables.
\begin{definition}
\thlabel{def:category-of-hyperdoctrines}
We call a functor $F: \FinSet \to \DLat$ a \emph{Lindenbaum-Tarski functor} if the following hold:
\begin{enumerate}[label=(\roman*)]
\item for every $f: \n \to \m$ the homomorphism $F_f: F_n \to F_m$ has a left adjoint $F_f^*: F_m \to F_n$, such that $F_f^* \dashv F_f$ satisfies the Frobenius condition;
\item the image under $F$ of every pushout in $\FinSet$ satisfies the Beck-Chevalley condition.
\end{enumerate}
The latter means that given a pushout as on the left in the diagram below, the induced square on the right commutes.
\[
\begin{tikzcd}
\mathbf{a} \arrow[rrd, "\text{pushout}", phantom] &  & \mathbf{b} \arrow[ll, "f"'] \arrow[rd, "\Rightarrow", phantom] & F_a \arrow[d, "F_g^*"'] \arrow[rrd, "\text{commutes}", phantom] &  & F_b \arrow[ll, "F_f"'] \arrow[d, "F_h^*"] \\
\mathbf{c} \arrow[u, "g"]                         &  & \mathbf{d} \arrow[ll, "k"] \arrow[u, "h"']                     & F_c                                                             &  & F_d \arrow[ll, "F_k"]                    
\end{tikzcd}
\]
\thref{def:partial-natural-transformation} and \thref{def:morphism-partial-natural-transformations} can easily be restated in terms of Lindenbaum-Tarski functors, so we view them as definitions for this setting as well. Note that a partial spectral map $f: X \dashrightarrow Y$ (with compact open domain) corresponds to a lattice homomorphism $f^{-1}: \KO(Y) \to \KO(X)$ that may not preserve the top element.

The \emph{2-category of Lindenbaum-Tarski functors} $\LTFuncBi$ has Lindenbaum-Tarski functors as objects, as 1-cells partial natural transformations that satisfy weak Beck-Chevalley and as 2-cells morphisms of partial natural transformations.

The \emph{category of Lindenbaum-Tarksi functors} $\LTFunc$ has Lindenbaum-Tarksi functors as objects and the arrows are natural transformations that satisfy the Beck-Chevalley condition.
\end{definition}
\begin{corollary}
\thlabel{cor:typespacefunc-dual-to-ltfunc}
The duality $\DLat^\op \simeq \Spec$ induces a 2-equivalence $\LTFuncBi^\op \simeq \TypeSpaceFuncBi$ that restricts to a duality $\LTFunc^\op \simeq \TypeSpaceFunc$.
\end{corollary}
\begin{remark}
\thlabel{rem:hyperdoctrines}
What we call a Lindenbaum-Tarski functor is a special case of the more general concept of a (coherent) hyperdoctrine (see \cite{lawvere_adjointness_2006} and \cite{coumans_generalising_2012}). Hyperdoctrines are more general in the sense that they admit any $\C^\op$ as domain where $\C$ is finitely complete. Besides that the 1-cells and 2-cells in our setting are different from those in \cite{coumans_generalising_2012}.
\end{remark}

\section{Recovering a theory}
\label{sec:recovering-a-theory}
The goal of this section is to construct a 2-functor $\Th(-)$ that will recover a theory from a type space functor. This 2-functor will turn out to be the inverse of $\S(-)$. In particular, $\Th(\S(T))$ will essentially be the Morleyisation of $T$.

The key ingredient in linking the theory we build to semantics is the Deligne completeness theorem (\thref{fact:deligne-completeness-theorem}). For this we need to introduce a deduction-system, which we formulate as a sequent calculus.
\begin{definition}
\thlabel{def:sequent-theory}
A \emph{(coherent) sequent} is an assertion of the form
\[
\phi \vdash_{\bar{x}} \psi,
\]
where $\phi$ and $\psi$ are coherent formulas and $\bar{x}$ is a finite tuple of variables, such that the free variables in $\phi$ and $\psi$ are in $\bar{x}$.
\end{definition}
A sequent $\phi \vdash_{\bar{x}} \psi$ corresponds to the h-inductive sentence $\forall \bar{x}(\phi(\bar{x}) \to \psi(\bar{x}))$. So a theory is the same thing as a set of sequents.
\begin{definition}
\thlabel{def:theory-of-type-space-functor}
Let $F$ be a type space functor, we define the \emph{internal logic} of $F$ as follows. The signature $\L(F)$ has an $n$-ary relation symbol $R_U$ for every $U \in \KO(F_n)$, for all $n \geq 0$.

An $n$-ary $\L(F)$-formula $\phi$ will be interpreted as some compact open $\llbracket \phi \rrbracket \subseteq F_n$, defined by induction on the complexity of the formula:
\begin{itemize}
\item $\llbracket \top \rrbracket = F_0$ and $\llbracket \bot \rrbracket = \emptyset$ ($\bot$ is considered as an $n$-ary formula for all $n$);
\item $\llbracket R_U(x_{f(1)}, \ldots, x_{f(n)}) \rrbracket = F_f^{-1}(U)$, for all $f: \n \to \m$ and $U \in \KO(F_n)$;
\item $\llbracket x_i = x_j \rrbracket = F_e(F_n)$, where $1 \leq i,j \leq n$ and $e: \n \to \n$ sends $i$ to $j$ and is the identity everywhere else;
\item $\llbracket \phi \wedge \psi \rrbracket = \llbracket \phi \rrbracket \cap \llbracket \psi \rrbracket$;
\item $\llbracket \phi \vee \psi \rrbracket = \llbracket \phi \rrbracket \cup \llbracket \psi \rrbracket$;
\item $\llbracket \exists x_{n+1} \phi(\bar{x}, x_{n+1}) \rrbracket = F_f(\llbracket \phi(\bar{x}, x_{n+1}) \rrbracket)$, where $f: \n \to \mathbf{n+1}$ is the inclusion.
\end{itemize}
Finally, we define the \emph{theory of $F$} as follows:
\[
\Th(F) = \{\phi \vdash_{\bar{x}} \psi : \llbracket \phi \rrbracket \subseteq \llbracket \psi \rrbracket\}.
\]
\end{definition}
\begin{remark}
\thlabel{rem:th-after-s}
Because the compact open subsets of $\S_n(T)$ correspond to formulas in the language of $T$, every $R_U \in \L(\S(T))$ is of the form $R_{[\phi(\bar{x})]}$ for some $\phi(\bar{x})$. We may assume that all variables in $\bar{x}$ actually appear in $\phi(\bar{x})$.
\end{remark}
The main result of this section is the following theorem.
\begin{theorem}
\thlabel{thm:theory-of-type-space-functor-vs-models}
Let $F$ be a type space functor, then for $\L(F)$-formulas $\phi$ and $\psi$:
\[
\Th(F) \models \forall \bar{x}(\phi(\bar{x}) \to \psi(\bar{x}))
\quad \text{if and only if} \quad
\llbracket \phi \rrbracket \subseteq \llbracket \psi \rrbracket.
\]
\end{theorem}
\begin{corollary}
\thlabel{cor:theory-of-type-space-functor-every-formula-is-relation}
Let $F$ be a type space functor, and let $\phi$ be an $\L(F)$-formula. Then $\phi$ is equivalent to $R_{\llbracket \phi \rrbracket}$, modulo $\Th(F)$.
\end{corollary}
\begin{proof}
Straightforward induction on the complexity of $\phi$, using \thref{thm:theory-of-type-space-functor-vs-models}.
\end{proof}
\begin{definition}
\thlabel{def:rest-of-th}
For a partial natural transformation $(\beta, k): F \dashrightarrow F'$, we define the interpretation $\Th(\beta, k): \Th(F') \to \Th(F)$ of arity $k$ as follows. For each compact open $U \subseteq F'_n$ we let
\[
\Th(\beta, k)(R_U) = R_{\beta_n^{-1}(U)},
\]
and we set $\Th(\beta, k)(x = y) = R_{\beta_2^{-1}(\llbracket x = y \rrbracket)}$.

For a morphism of partial natural transformations $\Theta: (\beta, k) \to (\beta', k')$ we define the morphism of interpretations $\Th(\Theta): \Th(\beta, k) \to \Th(\beta', k')$ as $R_{\Theta}$.
\end{definition}
\begin{lemma}
\thlabel{lem:properties-functor-th}
With the notation as in \thref{def:rest-of-th}, we have:
\begin{enumerate}[label=(\roman*)]
\item $\Th(\S(\Gamma, k)): \Th(\S(T)) \to \Th(\S(T'))$ is defined by
\[
R_{[\phi(\bar{x})]} \mapsto R_{[\Gamma(\phi(\bar{x}))]},
\]
where all variables in $\bar{x}$ appear in $\phi(\bar{x})$, and $\Th(\S(\Gamma, k))(x = y)$ is $R_{[\Gamma(x = y)]}$;
\item $\llbracket \Th(\beta, k)(\phi) \rrbracket = \beta_n^{-1}(\llbracket \phi \rrbracket)$.
\end{enumerate}
\end{lemma}
\begin{proof}
The proof of (i) is writing out definitions, where \thref{lem:properties-functor-s}(i) and \thref{rem:th-after-s} are useful. The proof of (ii) is a straightforward induction argument, where each step is writing out definitions.
\end{proof}
\begin{corollary}
\thlabel{cor:functor-th}
The operations defined in \thref{def:theory-of-type-space-functor} and \thref{def:rest-of-th} define a 2-functor $\Th: \TypeSpaceFuncBi \to \CohTheoryBi^\op$, which restricts to a functor $\Th: \TypeSpaceFunc \to \CohTheory^\op$.
\end{corollary}
\begin{proof}
From \thref{lem:properties-functor-th} and \thref{thm:theory-of-type-space-functor-vs-models} it follows that $\Th(\beta, k)$ is indeed an interpretation. All that is left to check is that a natural transformation that satisfies the Beck-Chevalley condition does indeed yield a strong interpretation.

So let $\beta: F \to F'$ be a natural transformation satisfying the Beck-Chevalley condition, and let $e: \mathbf{2} \to \mathbf{2}$ be given by $e(2) = e(1) = 1$, then:
\[
\beta_2^{-1}(\llbracket x = y \rrbracket) = \beta_2^{-1}(F_e'(F_2')) = F_e(\beta_2^{-1}(F_2')) = F_e(F_2) = \llbracket x = y \rrbracket.
\]
So $\Th(\beta)(x = y)$ is $R_{\llbracket x = y \rrbracket}$, which by \thref{cor:theory-of-type-space-functor-every-formula-is-relation} is equivalent to $x = y$.
\end{proof}
To prove \thref{thm:theory-of-type-space-functor-vs-models} we need to introduce the deduction-system for coherent logic \cite[Section D1.3]{johnstone_sketches_2002_v2}. For ease of reference, and to introduce notation, we recall the rules and axioms here.
\begin{definition}
\thlabel{def:deduction-system}
The deduction-system for coherent logic has the following rules and axioms. Throughout, $\bar{x}$ denotes the tuple $(x_1, \ldots, x_n)$ and $\phi$, $\psi$ and $\chi$ are arbitrary coherent formulas.
\begin{labeling}{(distributivity)}

\item[(identity)] The axiom: $\phi \vdash_{{\bar{x}}} \phi.$

\item[(substitution)] Let $f: \n \to \m$ be any function, then we have the rule:
\begin{prooftree}
\AxiomC{$\phi(x_1, \ldots, x_n) \vdash_{(x_1, \ldots, x_n)} \psi(x_1, \ldots, x_n)$}
\UnaryInfC{$\phi(x_{f(1)}, \ldots, x_{f(n)}) \vdash_{(x_1, \ldots, x_m)} \psi(x_{f(1)}, \ldots, x_{f(n)})$}
\end{prooftree}

\item[(cut)] The rule:
\begin{prooftree}
\AxiomC{$\phi \vdash_{\bar{x}} \psi$}
\AxiomC{$\psi \vdash_{\bar{x}} \chi$}
\BinaryInfC{$\phi \vdash_{\bar{x}} \chi$}
\end{prooftree}

\item[(equality1)] For all $1 \leq i \leq n$, the axiom: $\top \vdash_{\bar{x}} x_i = x_i.$

\item[(equality2)] For all $1 \leq i,j \leq n$, the axiom:
$
(x_i = x_j) \wedge \phi(x_1, \ldots, x_i, \ldots, x_n) \vdash_{\bar{x}} \phi(x_1, \ldots, x_j, \ldots, x_n).
$

\item[(conjunction)] The axioms:
$
\phi \vdash_{\bar{x}} \top,
\phi \wedge \psi \vdash_{\bar{x}} \phi,
\phi \wedge \psi \vdash_{\bar{x}} \psi,
$
and the rule:
\begin{prooftree}
\AxiomC{$\phi \vdash_{\bar{x}} \psi$}
\AxiomC{$\phi \vdash_{\bar{x}} \chi$}
\BinaryInfC{$\phi \vdash_{\bar{x}} \psi \wedge \chi$}
\end{prooftree}

\item[(disjunction)] The axioms:
$
\bot \vdash_{\bar{x}} \phi,
\phi \vdash_{\bar{x}} \phi \vee \psi,
\psi \vdash_{\bar{x}} \phi \vee \psi,
$
and the rule:
\begin{prooftree}
\AxiomC{$\phi \vdash_{\bar{x}} \chi$}
\AxiomC{$\psi \vdash_{\bar{x}} \chi$}
\BinaryInfC{$\phi \vee \psi \vdash_{\bar{x}} \chi$}
\end{prooftree}

\item[($\exists$-quantifier)] The double rule:
\begin{prooftree}
\AxiomC{$\phi(\bar{x}, x_{n+1}) \vdash_{\bar{x}, x_{n+1}} \psi(\bar{x})$}
\doubleLine
\UnaryInfC{$\exists x_{n+1} \phi(\bar{x}, x_{n+1}) \vdash_{\bar{x}} \psi(\bar{x})$}
\end{prooftree}

\item[(distributivity)] The axiom:
$
\phi \wedge (\psi \vee \chi) \vdash_{\bar{x}} (\phi \wedge \psi) \vee (\phi \wedge \chi).
$

\item[(Frobenius)] The axiom:
$
\phi(\bar{x}) \wedge \exists x_{n+1} \psi(\bar{x}, x_{n+1}) \vdash_{\bar{x}} \exists x_{n+1}(\phi(\bar{x}) \wedge \psi(\bar{x}, x_{n+1})).
$
\end{labeling}
\end{definition}
Note that the converse of the (distributivity) and (Frobenius) rules can be derived from this system \cite[page 832]{johnstone_sketches_2002_v2}. We also did not mention terms in our (substitution) rule, because our signature is purely relational.

It should be clear that the deduction-system for coherent logic is sound for classical semantics, in fact it is a subset of the usual classical deduction-system. It turns out that it is also complete, and we get the following fact, which appears as \cite[Corollary IX.11.3]{maclane_sheaves_1994}.
\begin{fact}[Deligne completeness theorem]
\thlabel{fact:deligne-completeness-theorem}
Let $T$ be a coherent theory. A sequent $\phi(\bar{x}) \vdash_{\bar{x}} \psi(\bar{x})$ can be deduced from $T$ in the deduction-system for coherent logic if and only if $T \models \forall \bar{x}(\phi(\bar{x}) \to \psi(\bar{x}))$.
\end{fact}
\begin{lemma}
\thlabel{lem:substitution-of-variables-interpretation-type-space-functor}
Let $F$ be a type space functor, and let $\phi$ be an $n$-ary formula in $\L(F)$. Then for every $f: \n \to \m$, we have
\[
F_f^{-1}(\llbracket \phi(x_1, \ldots, x_n) \rrbracket) = \llbracket \phi(x_{f(1)}, \ldots, x_{f(n)}) \rrbracket.
\]
\end{lemma}
\begin{proof}
By induction on the complexity of the formula. The non-trivial cases are the equality symbol and existential quantification. In both cases we use that $F$ satisfies Beck-Chevalley, by constructing the pushout of either $f$ and $e: \n \to \n$ or $f$ and the inclusion $\n \to \mathbf{n+1}$.
\end{proof}
\begin{proposition}
\thlabel{prop:theory-of-type-space-functor-is-sound}
Let $F$ be a type space functor, then $\Th(F)$ is closed under deductions in the deduction-system for coherent logic. That is, if $\phi$ and $\psi$ are $\L(F)$-formulas, then we can deduce $\phi \vdash_{\bar{x}} \psi$ from $\Th(F)$ precisely when $\llbracket \phi \rrbracket \subseteq \llbracket \psi \rrbracket$.
\end{proposition}
\begin{proof}
By definition we have that $\phi \vdash_{\bar{x}} \psi$ is in $\Th(F)$ if $\llbracket \phi \rrbracket \subseteq \llbracket \psi \rrbracket$. For the converse, we prove that $\Th(F)$ contains all the axioms mentioned in \thref{def:deduction-system} and that it is closed under all the rules mentioned in that definition.

The rules (identity), (cut), (equality1), (conjunction), (disjunction) and (distributivity) follow directly from the definition of $\Th(F)$. The (substitution) rule follows from \thref{lem:substitution-of-variables-interpretation-type-space-functor}.

For ($\exists$-quantifier) and (Frobenius) we let $f: \n \to \mathbf{n+1}$ be the inclusion. Then the former is precisely the adjunction $F_f \dashv F_f^{-1}$, and the latter follows from the fact that this adjunction satisfies the Frobenius condition (\thref{prop:frobenius-for-open-spectral-maps}).

Finally, we have to check (equality2). Let $e: \n \to \n$ be the map that sends $i$ to $j$ and is the identity everywhere else. We have $\llbracket x_i = x_j \rrbracket = F_e(F_n)$ and $\llbracket \phi(x_1, \ldots, x_j, x_j, \ldots, x_n) \rrbracket = F_e^{-1}(\llbracket \phi(x_1, \ldots, x_i, x_j, \ldots, x_n) \rrbracket)$. So, we have
\[
\llbracket (x_i = x_j) \wedge \phi(x_1, \ldots, x_i, x_j, \ldots, x_n) \rrbracket =
F_e(F_n) \cap \llbracket \phi(x_1, \ldots, x_i, x_j, \ldots, x_n) \rrbracket.
\]
By the Frobenius condition (\thref{prop:frobenius-for-open-spectral-maps}) this is equal to
\[
F_e(F_n \cap F_e^{-1}(\llbracket \phi(x_1, \ldots, x_i, x_j, \ldots, x_n) \rrbracket)) =
F_e(\llbracket \phi(x_1, \ldots, x_j, x_j, \ldots, x_n) \rrbracket).
\]
Note that $\llbracket \phi(x_1, \ldots, x_j, x_j, \ldots, x_n) \rrbracket = F_e^{-1}(\llbracket \phi(x_1, \ldots, x_j, x_j, \ldots, x_n) \rrbracket)$, either using \thref{lem:substitution-of-variables-interpretation-type-space-functor} or the fact that $e$ is idempotent. So we conclude that indeed
\begin{align*}
&\llbracket (x_i = x_j) \wedge \phi(x_1, \ldots, x_i, x_j, \ldots, x_n) \rrbracket &= \\
&F_e(\llbracket \phi(x_1, \ldots, x_j, x_j, \ldots, x_n) \rrbracket) &= \\
&F_e(F_e^{-1}(\llbracket \phi(x_1, \ldots, x_j, x_j, \ldots, x_n) \rrbracket)) &\subseteq \\
&\llbracket \phi(x_1, \ldots, x_j, x_j, \ldots, x_n)) \rrbracket.
\end{align*}
\end{proof}
\begin{proof}[Proof of \thref{thm:theory-of-type-space-functor-vs-models}]
Combine \thref{prop:theory-of-type-space-functor-is-sound} and \thref{fact:deligne-completeness-theorem}.
\end{proof}
\section{The 2-equivalence}
\label{sec:the-2-equivalence}
We recall the statement of the main theorem. Note that both \thref{cor:duality-first-order-theory-boolean-type-space} and \thref{cor:duality-theory-ltfunc} follow directly, using \thref{fact:properties-type-space-functor-of-a-theory}(ii) and \thref{cor:typespacefunc-dual-to-ltfunc} respectively.
\begin{repeated-theorem}[\thref{thm:equivalence-2-categories}]
The 2-functors
\[
S: \CohTheoryBi^\op \rightleftarrows \TypeSpaceFuncBi: \Th
\]
form a 2-equivalence of 2-categories, which restricts to an equivalence
\[
\S: \CohTheory^\op \rightleftarrows \TypeSpaceFunc: \Th
\]
\end{repeated-theorem}
\begin{proof}
We construct 2-natural isomorphisms $\Th(\S(T)) \cong T$ and $\S(\Th(F)) \cong F$. Since these isomorphisms will also live in $\CohTheory$ and $\TypeSpaceFunc$ respectively, this proves both equivalences at the same time. Throughout this proof, when we write $\phi(\bar{x})$ for some formula $\phi$, we mean that all variables in $\bar{x}$ actually appear in $\phi$.

\textbf{The isomorphism $\Th(\S(T)) \cong T$.} We start by defining a strong interpretation $\Gamma_T: \Th(\S(T)) \to T$ as follows: for $R_{[\phi(\bar{x})]} \in \L(\S(T))$ we set $\Gamma_T(R_{[\phi(\bar{x})]}) = \phi(\bar{x})$. Note that by \thref{rem:th-after-s} this covers all relation symbols in $\L(\S(T))$. For a formula $\psi(\bar{x})$ in $\L(\S(T))$, we easily see by induction that $[\Gamma_T(\psi(\bar{x}))] = \llbracket \psi(\bar{x}) \rrbracket$. So by \thref{thm:theory-of-type-space-functor-vs-models}, we have
\begin{align*}
&\Th(\S(T)) \models \forall \bar{x}(\psi(\bar{x}) \to \chi(\bar{x})) &\Longleftrightarrow \\
&\llbracket \psi \rrbracket \subseteq \llbracket \chi \rrbracket &\Longleftrightarrow \\
&[\Gamma_T(\psi)] \subseteq [\Gamma_T(\chi)] &\Longleftrightarrow \\
&T \models \forall \bar{x}(\Gamma_T(\psi)(\bar{x}) \to \Gamma_T(\chi)(\bar{x})).
\end{align*}
So $\Gamma_T$ is indeed a strong interpretation, and in particular an isomorphism.

Naturality in 1-cells for $\Gamma_T$ follows from \thref{lem:properties-functor-th}. For $(\Delta, k): T \to T'$:
\[
\Delta(\Gamma_T(R_{[\phi(\bar{x})]})) =
\Delta(\phi(\bar{x})) =
\Gamma_{T'}(R_{[\Delta(\phi(\bar{x}))]}) =
\Gamma_{T'}(\Th(\S(\Delta, k))(R_{[\phi(\bar{x})]})).
\]
Naturality in 2-cells amounts to writing down the relevant square and then writing out the definitions.

\textbf{The isomorphism $\S(\Th(F)) \cong F$.} We define $\beta_F: \S(\Th(F)) \to F$ as follows. Let $n \geq 0$, and let $p \in \S_n(\Th(F))$. We claim that
\[
\U_p = \{U \in \KO(F_n) : R_U \in p\}
\]
is a prime filter on $\KO(F_n)$. Clearly $\emptyset \not \in \U_p$, because $p$ is consistent. For $U \subseteq V$, we have by definition of $\Th(F)$ that $\Th(F) \models \forall \bar{x}(R_U(\bar{x}) \to R_V(\bar{x}))$. So if $U \in \U_p$, we also have $V \in \U_p$. Similarly, we have that $R_{U \cap V}$ and $R_{U \cup V}$ are equivalent to $R_U \wedge R_V$ and $R_U \vee R_V$ respectively, modulo $\Th(F)$. From which it quickly follows that $\U_p$ is a prime filter.

By the Stone duality $\U_p$ corresponds to an element in $F_n$. With some abuse of notation we will identify this element with $\U_p$, and we set $\beta_{F,n}(p) = \U_p$.

We check that $\beta_{F,n}$ is a homeomorphism, and that it is natural $n$. Since homeomorphisms are spectral maps, and natural isomorphisms always satisfy Beck-Chevalley, we then know that $\beta_F$ is indeed an isomorphism in $\TypeSpaceFuncBi$ and $\TypeSpaceFunc$. To check that $\beta$ is a 2-natural transformation, one only needs write out the definitions to see that the relevant diagrams commute, which we omit.

\underline{Homeomorphism.} We first prove that $\beta_{F,n}$ is continuous. Let $U$ be a compact open set in $F_n$, then
\[
p \in \beta_{F,n}^{-1}(U) \;\Longleftrightarrow\;
\beta_{F,n}(p) \in U \;\Longleftrightarrow\;
U \in \U_p \;\Longleftrightarrow\;
R_U \in p \;\Longleftrightarrow\;
p \in [R_U].
\]
So $\beta_{F,n}^{-1}(U) = [R_U]$, and so we see that $\beta_{F,n}$ is indeed continuous.

Next we define a map $g: F_n \to \S(\Th(F_n))$ that will be the inverse of $\beta_{F,n}$. Let $q \in F_n$, and define the set of first-order formulas:
\[
P_q = \Th(F) \cup \{R_U(\bar{x}) : q \in U\} \cup \{\neg R_V(\bar{x}) : q \not \in V\}.
\]
We claim that $P_q$ is consistent. Clearly, $\Th(F)$ is consistent, because otherwise $F_n$ would be empty. If $P_q$ would be inconsistent, then by compactness there would be $U_1, \ldots, U_k, V_1, \ldots, V_m \in \KO(F_n)$, such that
\[
\Th(F) \models \neg \exists \bar{x}(R_{U_1}(\bar{x}) \wedge \ldots \wedge R_{U_k}(\bar{x}) \wedge \neg R_{V_1}(\bar{x}) \wedge \ldots \wedge \neg R_{V_m}(\bar{x})),
\]
which is equivalent to
\[
\Th(F) \models \forall \bar{x}(R_{U_1}(\bar{x}) \wedge \ldots \wedge R_{U_k}(\bar{x}) \to R_{V_1}(\bar{x}) \vee \ldots \vee R_{V_m}(\bar{x})).
\]
By \thref{thm:theory-of-type-space-functor-vs-models}, this would mean that $U_1 \cap \ldots \cap U_k \subseteq V_1 \cup \ldots \cup V_k$. However, by construction we have $q \in U_1 \cap \ldots \cap U_k$ and $q \not \in V_1 \cup \ldots \cup V_k$. We thus arrive at a contradiction and conclude that $P_q$ must be consistent. Let $\bar{a}$ in some model $M$ of $\Th(F)$ realise $P_q$, then we set $g(q) = \ctp^M(\bar{a})$. This is well-defined, because every $n$-ary formula in $\L(F)$ is equivalent to an $n$-ary relation symbol (by \thref{cor:theory-of-type-space-functor-every-formula-is-relation}), and for every $n$-ary relation symbol, $P_q$ contains that symbol or its negation. In particular, $\ctp^M(a)$ is determined by $\{R_U(x) : q \in U\}$, which is essentially the prime filter corresponding to $q$ (under the Stone duality). So $g$ is indeed the inverse function for $\beta_{F,n}$. We are left to check that $g$ is continuous. For this we let $[R_U]$ be any compact open subset of $\S_n(\Th(F))$ (by a combination of \thref{fact:properties-type-space-functor-of-a-theory}(i) and \thref{cor:theory-of-type-space-functor-every-formula-is-relation}, every compact open subset is of that form), then we have:
\[
q \in g^{-1}([R_U]) \;\Longleftrightarrow\;
g(q) \in [R_U] \;\Longleftrightarrow\;
R_U \in g(q) \;\Longleftrightarrow\;
q \in U.
\]
So we have $g^{-1}([R_U]) = U$, and we conclude that $g$ is indeed continuous and thus that $\beta_{F,n}$ is a homeomorphism.

\underline{Naturality in $n$.} Let $f: \n \to \m$ be any function. We again abuse notation, and identify elements in the spectral spaces with their corresponding prime filters. Writing out definitions we obtain
\begin{align*}
F_f(\beta_{F,m}(p)) &= \{U \in \KO(F_n) : R_{F_f^{-1}(U)}(x_1, \ldots, x_m) \in p\}, \\
\beta_{F,n}(\S_f(\Th(F))(p)) &= \{ U \in \KO(F_n) : R_U(x_{f(1)}, \ldots, x_{f(n)}) \in p \}.
\end{align*}
By definition $\llbracket R_{F_f^{-1}(U)}(x_1, \ldots, x_m) \rrbracket = \llbracket R_U(x_{f(1)}, \ldots, x_{f(n)}) \rrbracket$, so by \thref{thm:theory-of-type-space-functor-vs-models} these are equivalent formulas and hence $F_f(\beta_{F,m}(p)) = \beta_{F,n}(\S_f(\Th(F))(p))$.
\end{proof}
\section{Classifying toposes}
\label{sec:connection-with-classifying-toposes}
There is an obvious connection between $\CohTheoryBi$ and the 2-category $\Topos$ of (Grothendieck) toposes, via classifying toposes. In this section we will briefly sketch this connection in terms of a pseudofunctor $\Set[-]: \CohTheoryBi^\op \to \Topos$. It is beyond the scope of this paper to study this pseudofunctor any further. We do provide a quick example (\thref{ex:morita-equivalence-vs-bi-interpretable}) why it cannot be a biequivalence, even after restricting its codomain.

We briefly recall the definition of a classifying topos, for more information we refer to \cite[Section D3]{johnstone_sketches_2002_v2} and \cite{caramello_theories_2017}. The \emph{classifying topos of $T$} is defined to be the topos $\Set[T]$ such that there is a natural equivalence
\[
T\Mod(\E) \simeq \Hom(\E, \Set[T]),
\]
where $T\Mod(\E)$ is the category of models of $T$ internal to $\E$. It follows that there must be a \emph{generic model} $G_T$ in $\Set[T]$, corresponding to the identity morphism, such that for any geometric morphism $f: \E \to \Set[T]$ the corresponding model is given by $f^*(G_T)$.

\begin{definition}
\thlabel{def:pseudofunctor-classifying-topos}
Every coherent theory $T$ has a classifying topos $\Set[T]$. We extend this to a pseudofunctor $\Set[-]: \CohTheoryBi^\op \to \Topos$ by describing what it does on 1-cells and 2-cells.

Let $(\Gamma, k): T \to T'$ be an interpretation, and let $M$ be a model of $T'$ in some topos $\E$. Then we can carry out the construction from \thref{rem:interpretation-models} internally in $\E$ to obtain a model $\Gamma^*(M)$ of $T$ in $\E$. In particular, if we consider $G_{T'}$ in $\Set[T']$, then we obtain a model $\Gamma^*(G_{T'})$ of $T$ in $\Set[T']$. Under the equivalence
\begin{align}
\label{eq:classifying-topos-set-t-prime-set-t}
T\Mod(\Set[T']) \simeq \Hom(\Set[T'], \Set[T]),
\end{align}
this model corresponds to a geometric morphism $\Set[\Gamma, k]: \Set[T'] \to \Set[T]$.

Let $\theta: (\Gamma, k) \to (\Gamma', k')$ be a morphism of interpretations from $T$ to $T'$. Carrying out the construction from \thref{rem:morphism-of-interpretations-gives-homomorphism-of-models} internally in $\Set[T']$ we get an internal homomorphism $f_\theta: \Gamma^*(G_{T'}) \to \Gamma'^*(G_{T'})$. Under the equivalence in (\ref{eq:classifying-topos-set-t-prime-set-t}), this corresponds to a natural transformation $\Set[\theta]: \Set[\Gamma, k] \to \Set[\Gamma', k']$.
\end{definition}
\begin{remark}
\thlabel{rem:different-notions-of-interpretation}
In the topos-theoretic setting an interpretation (of coherent theories) is often defined as a coherent functor between the syntactic categories, see e.g.\ \cite[Definition 1.4.12]{caramello_theories_2017}. This forces the equality symbol to be interpreted as the equality symbol on some definable set. This is thus weaker than our definition of interpretation, where it could be a definable equivalence relation in general. So the temptation to define $\Set[-]$ via syntactic categories must be resisted.
\end{remark}
Even if we restrict codomain of $\Set[-]$ to coherent toposes (those toposes that classify a coherent theory), it cannot be a biequivalence.
\begin{definition}[{\cite[Definition 2.2.1]{caramello_theories_2017}}]
\thlabel{def:morita-equivalence}
We call two theories $T$ and $T'$ \emph{Morita-equivalent} if their classifying toposes are equivalent: $\Set[T] \simeq \Set[T']$. Equivalently, this means that they have equivalent categories of models in every topos $\E$, natural in $\E$. That is, we have $T\Mod(\E) \simeq T'\Mod(\E)$, natural in $\E$.
\end{definition}
For those that are more familiar with model-theoretic constructs, there is another way of characterising Morita-equivalence. Recall that for a theory $T$ the \emph{syntactic category} $\C_T$ has formulas as objects and definable functions as arrows. One can then form the so-called \emph{pretopos completion} $P(\C_T)$ (see e.g.\ \cite[Remark 1.3.18(a)]{caramello_theories_2017}). Two theories $T$ and $T'$ are then Morita-equivalent if and only if the corresponding pretopos completions are equivalent as categories, so $P(\C_T) \simeq P(\C_{T'})$ (see e.g.\ \cite[Corollary 4.9]{tsementzis_syntactic_2015}). Forming the pretopos completion amounts to adding disjoint unions and coequalizers of equivalence relations. So this is very similar to the $(-)^\text{eq}$-construction in model theory (see e.g.\ \cite[page 151]{hodges_model_1993}), only in constructing $P(\C_T)$ we also add disjoint unions. This last difference will be precisely the obstacle, as we will see below.

It is well-known that Morita-equivalence is weaker than being bi-interpretable. We provide a basic counterexample. More elaborate examples can be found in \cite{caramello_lattice-ordered_2014, caramello_morita-equivalence_2015}.
\begin{example}
\thlabel{ex:morita-equivalence-vs-bi-interpretable}
For a distributive lattice $L$, let $T(L)$ be the (propositional) theory of that lattice. That is, we have a propositional symbol $R_a$ for each $a \in L$, and
\[
T(L) = \{\neg \exists x(x = x)\} \cup \{R_a \to R_b : a \leq b \text{ in } L\}.
\]
Let $\O(\R)$ be the lattice of open subsets of $\R$, and let $\B \subseteq \O(\R)$ be a countable basis. Then $\Set[T(\O(\R))]$ and $\Set[T(\B)]$ are both the topos of sheaves on $\R$. So $T(\O(\R))$ and $T(\B)$ are Morita-equivalent. However, they can never be bi-interpretable. To see this, let $(\Gamma, k): T(\O(\R)) \rightleftarrows T(\B): (\Delta, \ell)$ be any two interpretations. Then by the pigeonhole principle, there must be different opens $U, V \subseteq \R$ such that $\Gamma(R_U) = \Gamma(R_V)$, hence $\Delta \Gamma(R_U) = \Delta \Gamma(R_V)$. So $(\Delta \Gamma, k \ell)$ cannot be isomorphic to the identity interpretation, because that would mean that $R_U$ is equivalent to $\Delta \Gamma(R_U) = \Delta \Gamma(R_V)$, and hence to $R_V$.
\end{example}
If we restrict to a certain class of theories, these counterexamples go away.
\begin{definition}
\thlabel{def:disjoint-theory}
Call a theory $T$ \emph{positively disjoint} if there are formulas $\phi(\bar{x})$ and $\psi(\bar{x})$ such that $T \models \exists \bar{x}\phi(\bar{x}) \wedge \exists \bar{x} \psi(\bar{x})$ and $T \models \forall \bar{x} \neg (\phi(\bar{x}) \wedge \psi(\bar{x}))$.
\end{definition}
From \cite[Proposition 5.11]{mceldowney_morita_2020} we essentially get the following fact. The fact independently also appears in \cite{washington_equivalence_2018}.
\begin{fact}
\thlabel{fact:disjoint-theory-morita-equivalence}
If $T$ and $T'$ are positively disjoint theories, then they are Morita equivalent if and only if they are bi-interpretable.
\end{fact}
We make two remarks about this fact. The first is that they require that $T \models \exists x y(x \neq y)$ and similar for $T'$, instead of being positively disjoint. The inequality symbol is a problem for positive logic, hence the notion of positively disjoint theories. Their proof is easily adjusted. The point is that we can encode so-called ``coproduct sorts''.

The second remark is that their notion of Morita equivalence is based on the syntactic definition from \cite{barrett_morita_2016} and it is for first-order theories. In \cite{tsementzis_syntactic_2015} this is called \emph{T-Morita equivalence} and they establish a connection with coherent theories and classifying toposes. That is, they prove that for coherent theories T-Morita equivalence coincides with having the same classifying topos (which they call \emph{J-Morita equivalence}). They also prove that a first-order theory is T-Morita equivalent to its Morleyisation.
\begin{question}
\thlabel{q:classifying-topos-functor-on-positively-disjoint}
Aside from reflecting equivalences, what nice properties does $\Set[-]$ have when restricted to positively disjoint theories?
\end{question}

\bibliographystyle{alpha}
\bibliography{bibfile}


\end{document}